\documentclass[a4paper,12pt]{article}
\usepackage{amsmath}
\usepackage{amssymb}			
\usepackage{amsthm}\usepackage{latexsym}
\usepackage{amsmath}

\usepackage{amsthm}
\usepackage[dvipdfmx]{graphicx}
\usepackage{txfonts}
\usepackage{bm}
\usepackage{color}
\usepackage{epic,eepic}

\usepackage{rotating}

\usepackage{geometry}
\numberwithin{equation}{section}

\newtheorem{theorem}{Theorem}[section]
\newtheorem{proposition}[theorem]{Proposition}

\newtheorem{lemma}[theorem]{Lemma}
\newtheorem{remark}{Remark}[section]
\newtheorem{example}{Example}[section]

\newcommand{\memo}[1]{{\bf \small {\color{red}[MEMO:}} {\color{blue} #1} \ {\bf{\color{red}:end]} }}  %%% For work
\newcommand{\OMIT}[1]{{\bf [OMIT:} #1 \ {\bf --- end OMIT] }}  %%% For work
   \renewcommand{\OMIT}[1]{}            %%% For FINAL
\newcommand{\RR}{{\mathbb{R}}}
\newcommand{\ZZ}{{\mathbb{Z}}}
\newcommand{\vecone}{{\bf 1}}
\newcommand{\veczero}{{\bf 0}}
\newcommand{\dom}{{\rm dom\,}}

\newcommand{\suppp}{{\rm supp}\sp{+}}
\newcommand{\suppm}{{\rm supp}\sp{-}}

\newcommand{\unitvec}[1]{\bm{1}\sp{#1}}
\newcommand{\argmin}{\arg \min}
\newcommand{\conv}{\Box\,}

\newcommand{\todaye}{\the\year/\the\month/\the\day}

\newcommand{\Lnat}{{L$^{\natural}$}}
\newcommand{\Mnat}{{M$^{\natural}$}}

\newcommand{\Bvexb}{\mbox{\rm\bf (B-EXC)}}
\newcommand{\Bnvex}{\mbox{\rm (B$\sp{\natural}$-EXC)} }
\newcommand{\Bnvexb}{\mbox{\rm\bf (B$\sp{\natural}$-EXC)}}

\newcommand{\Mvexb}{\mbox{\rm\bf (M-EXC)}}

\newcommand{\Mnvexb}{\mbox{\rm\bf (M$\sp{\natural}$-EXC)}}

\newcommand{\YES}{Y \ }
\newcommand{\NO}{\quad \textbf{\textit{N}}}

\begin{document}

\title{On Basic Operations Related to 
\\
Network Induction of Discrete Convex Functions
%%\thanks{
%%This work was supported by CREST, JST, Grant Number JPMJCR14D2, Japan, and
%%JSPS KAKENHI Grant Number 26280004.}
}%title

\author{
Kazuo Murota%
\thanks{Department of Economics and Business Administration,
Tokyo Metropolitan University, 
Tokyo 192-0397, Japan, 
murota@tmu.ac.jp}
}%%author

\date{January 2020 / February 2020 / May 2020 / August 2020}
%%\date{December 2019 (Version \today)}

\maketitle

\begin{abstract}
Discrete convex functions are used in many areas,
including operations research, discrete-event systems, game theory, and economics.
The objective of this paper is to investigate
basic operations such as direct sum, splitting, and aggregation
that are related to network induction of
discrete convex functions as well as discrete convex sets.
Various kinds of discrete convex functions in discrete convex analysis
are considered 
such as
integrally convex functions,
{\rm L}-convex functions, {\rm M}-convex functions,
multimodular functions, and discrete midpoint convex functions.
\end{abstract}


\newpage
\tableofcontents
\newpage

%%\input{opnet1intro}

%% 2019-12-03 / 2019-12-25 / 2020-05-08

\section{Introduction}
\label{SCintro}

In matroid theory it is well known that a matroid is induced or transformed
by bipartite graphs through matchings
(\cite[Section 11.2]{Oxl11}, \cite[Section 8.2]{Wel76}).
Let $G$ be a bipartite graph with vertex bipartition consisting of
$N = \{ 1,2, \ldots, n \}$ and 
$M = \{ 1,2, \ldots, m \}$ as in Fig.~\ref{FGbiparmatroid} (a).
When a matroid $(N, \mathcal{I})$ is given on $N$ in terms of the family $\mathcal{I}$ of independent sets,
let $\mathcal{J}$ denote the collection of subsets of $M$ which can be matched in $G$
with an independent subset of $N$.
Then $(M, \mathcal{J})$ is a matroid, which is referred to as the 
matroid induced from $(N, \mathcal{I})$ by $G$.
We may regard this construction as a transformation of a matroid to another matroid.
If a free matroid is given on $N$, for example,
the matroid induced on $M$ is a transversal matroid.
In particular, 
a free matroid on $N$ is transformed to a partition matroid on $M$,
if the graph $G$ has a special structure
like Fig.~\ref{FGbiparmatroid} (b), where each vertex of $M$ has exactly one incident arc.
The union (or sum) operation for matroids
can also be understood as a transformation of this kind.
Given two matroids on $\{ 1,2, \ldots, n \}$ we consider 
a bipartite graph of the form of Fig.~\ref{FGbiparmatroid} (d),
in which the direct sum of the given matroids is associated with
the left vertex set $\{ 1,2, \ldots, n \} \cup \{ 1',2', \ldots, n' \}$,
and the induced matroid coincides with the union of the given matroids.
It is possible to generalize the above construction 
by replacing a bipartite graph with a general directed graph
and matchings with linkings;
see \cite[Section 11.2]{Oxl11} and \cite[Section 13.3]{Wel76}.

%%%  FIGURE %%%%%%%%%%%%%%%%%%
%%\input{fg1matroBipar}
\begin{figure}\begin{center}
 \includegraphics[width=0.8\textwidth,clip]{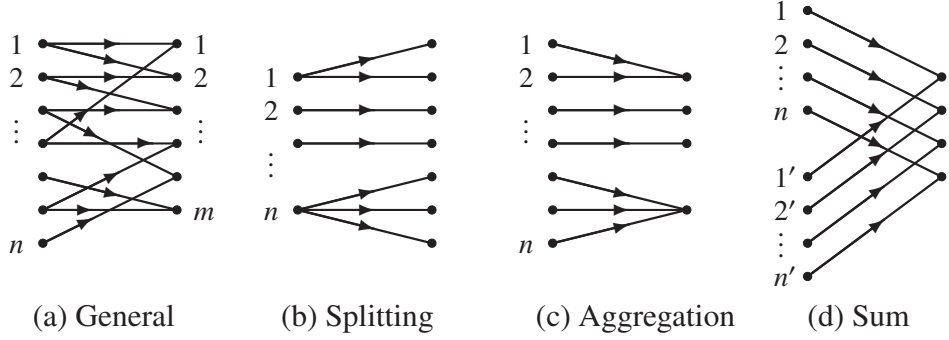}
\caption{Bipartite graphs for operations on discrete structures}
\label{FGbiparmatroid}
%%\hfill \memo{fg1matroBipar.tex}
\end{center}\end{figure}
%%%  FIGURE %%%%%%%%%%%%%%%%%%

In discrete convex analysis 
\cite{Fuj05book,Mdca98,Mdcasiam,Mbonn09,Mdcaeco16},
the transformation of matroids described above 
is generalized to a transformation of discrete convex sets and functions
by capacitated networks, which is called the network transformation.
The objective of this paper is to systematically investigate 
the network transformation,
together with the related basic operations, 
for discrete convex sets and functions.
The network transformation 
in discrete convex analysis 
is more general than
the transformation of matroids 
in the following two respects:
\begin{itemize}
\item
From $\{ 0, 1 \}$ to $\ZZ$:
A set family on the ground set $\{ 1,2, \ldots, n \}$ 
can be identified with a subset of $\{ 0, 1 \}\sp{n}$,
and hence the transformation of a matroid can be regarded as a transformation 
of a subset of $\{ 0, 1 \}\sp{n}$ to a subset of $\{ 0, 1 \}\sp{m}$.
A discrete convex set is a subset of $\ZZ\sp{n}$
that has some defining properties, and the network transformation 
of a discrete convex set amounts to 
a transformation of a subset $S$ of $\ZZ\sp{n}$ to a subset $T$ of $\ZZ\sp{m}$
via integral flows in an arc-capacitated network.
We are naturally interested in whether the resulting set $T$ 
is a discrete convex set of the same kind.

\item
From sets to functions:
A discrete convex set, which is a subset of $\ZZ\sp{n}$,
can be identified with its indicator function, which is equal to 0 on that set
and $+\infty$ elsewhere.
We generalize this by considering functions 
$f: \ZZ\sp{n} \to \RR \cup \{ +\infty \}$ 
that have certain discrete convexity properties.
The network transformation of a discrete convex function is 
defined via integral flows in a network with arc costs.
We are naturally interested in whether the resulting function 
$g: \ZZ\sp{m} \to \RR \cup \{ +\infty \}$ 
is a discrete convex function of the same kind.
\end{itemize}

The network transformation 
of a subset of $\ZZ\sp{n}$
is defined  (roughly) as follows.
For simplicity of presentation, 
we restrict ourselves to a bipartite network.
Let $G$ be a bipartite graph with vertex bipartition consisting of
$N = \{ 1,2, \ldots, n \}$ and 
$M = \{ 1,2, \ldots, m \}$ as in Fig.~\ref{FGbiparmatroid} (a),
and suppose that a nonnegative integer (upper) capacity is specified for each arc,
where the lower capacity is assumed to be zero.
When a set $S \subseteq \ZZ\sp{n}$ is given, 
let $T$ denote the collection of vectors $y \in \ZZ\sp{m}$ 
which can be linked from some $x \in S$ 
via an integer-valued flow meeting the capacity constraint.
If we interpret $S$ as a set of feasible supply vectors,
then the resulting set $T$ represents the set of demand vectors 
that can be realized by a feasible transportation scheme.
The {\em transformation} (or {\em induction}) by $G$ will mean 
the operation of obtaining $T$ from $S$.

It turns out to be convenient to single out two special types of bipartite graphs,
which are depicted in Fig.~\ref{FGbiparmatroid} (b) and (c).
In the graph in (b), each vertex of $M$ has exactly one incident arc,
and the transformation represented by such a graph will be called a {\em splitting}.
In the graph in (c), in contrast, each vertex of $N$ has exactly one incident arc,
and the transformation by such a graph will be called an {\em aggregation}.
While splitting and aggregation are special cases of the transformation by bipartite networks,
they are general enough in the sense
that the transformation by an arbitrary bipartite graph $G$ can be represented
as a composition of the transformation by a graph $G_{1}$ of type (b) 
followed by the transformation by a graph $G_{2}$ of type (c),
where $G_{1}$ and $G_{2}$ are obtained 
from $G$
(drawn as in Fig.~\ref{FGbiparmatroid})
by ``vertically cutting $G$ 
into left and right parts.'' 
The {\em Minkowski sum}
$S_{1}+S_{2} = \{ y \in \ZZ\sp{n} \mid y= x+ x', \ x \in S_{1}, \  x' \in S_{2} \}$
of sets $S_{1}$, $S_{2} \subseteq \ZZ\sp{n}$ 
is represented by the graph in Fig.~\ref{FGbiparmatroid} (d),
that is, the Minkowski sum can be represented as a combination of the direct sum
and aggregation operations.
Furthermore, 
it is known \cite{KMT07jump} that the transformation 
by a general capacitated network (to be defined in Section \ref{SCnettransset})
can be realized by a combination of splitting, aggregation,
and other basic operations.

The network transformation  of a function on $\ZZ\sp{n}$
is defined  (roughly) as follows.
We continue to refer to a bipartite graph $G$
in Fig.~\ref{FGbiparmatroid} (a), but
we now suppose that each arc is associated with a (convex) function 
to represent the cost of an integral flow in the arc.
When a function $f$ on $\ZZ\sp{n}$ is given,
we interpret $f(x)$ as the production cost of $x \in \ZZ\sp{n}$.
For $y \in \ZZ\sp{m}$, interpreted as a demand,
let $g(y)$ denote the minimum cost of an integral flow
that meets the demand $y$ by an appropriate choice of production $x$ 
and transportation scheme using an integer-valued flow.
The {\em transformation} (or {\em induction}) by $G$ will mean 
the operation of obtaining $g$ from $f$.

The special types of bipartite graphs in Fig.~\ref{FGbiparmatroid} (b) and (c)
continue to play the key role also for operations on functions.
The transformations of a function by the graphs in (b) and (c) are called a {\em splitting}
and an {\em aggregation} of the function, respectively.
As with the transformation of a discrete convex set,
the transformation of a function by an arbitrary bipartite graph can be represented
as a composition of the transformation by a graph of type (b) 
followed by the transformation by a graph of type (c).
For functions $f_{1}$ and $f_{2}$ on $\ZZ\sp{n}$,
their {\em convolution} 
\begin{equation*} %%\label{f1f2convdef0}
(f_{1} \conv f_{2})(y) =
 \inf\{ f_{1}(x) + f_{2}(x') \mid  y = x + x' \}
\qquad (y \in \ZZ\sp{n}) 
\end{equation*}
is represented by the graph in Fig.~\ref{FGbiparmatroid} (d),
that is, the convolution can be represented as a combination of the direct sum
and aggregation operations.
Furthermore, 
it is known \cite{KMT07jump} that the transformation of a function 
by a general network (to be defined in Section \ref{SCnettransfn})
can be realized by a combination of splitting, aggregation, and other basic operations.

Discrete convex functions treated in this paper
include
integrally convex functions 
\cite{FT90},
{\rm L}- and \Lnat-convex functions
\cite{FM00,Mdca98},
{\rm M}- and \Mnat-convex functions
\cite{Mstein96,Mdca98,MS99gp},
multimodular functions
\cite{Haj85},
globally and locally discrete midpoint convex functions 
\cite{MMTT20dmc},
and M- and \Mnat-convex functions on jump systems
\cite{Mmjump06,Mmnatjump19}.
It is noted that
``\Lnat'' and ``\Mnat''  should be pronounced as ``ell natural'' 
 and  ``em natural,'' respectively.  
L- and \Lnat-convex functions 
have applications in several different fields including
image processing,
auction theory, 
inventory theory, 
and scheduling
\cite{Che17,Mdcaeco16,Shi17L,SCB14}.
M- and \Mnat-convex functions 
find applications in game theory and economics
\cite{Mdcasiam,Mdcaeco16,MTcompeq03,ST15jorsj}
as well as in matrix theory
\cite[Chapter 5]{Mspr2000}.
Multimodular functions
have been used as a fundamental tool 
in the literature of queueing theory, discrete-event systems, and operations research
\cite{AGH00,AGH03,FHS17,GY94mono,Haj85,KS03,LY14,SW93markov,dWvS00,WS87netque,ZL10}.
Jump M- and \Mnat-convex functions find applications in several fields including
matching theory \cite{BK12,KST12cunconj,KT09evenf,Tak14fores}
and algebra \cite{Bra10halfplane}.
Integrally convex functions are
used in formulating discrete fixed point theorems
\cite{Iim10,IMT05,Yan09fixpt}, 
and designing solution algorithms for discrete systems of nonlinear equations
\cite{LTY11nle,Yan08comp}.
In game theory 
the integral concavity of payoff functions 
guarantees the existence of a pure strategy equilibrium 
in finite symmetric games \cite{IW14}.

This paper is intended to be a continuation of the recent paper \cite{Msurvop19},
which is the first systematic study of fundamental operations 
for various kinds of discrete convex functions
including multimodular functions and discrete midpoint convex functions.
While the paper \cite{Msurvop19}
dealt with basic operations such as restriction, projection, scaling, and convolution,
this paper focuses on operations 
related to the network transformation
including direct sum, splitting, and aggregation.
We mention that a systematic study of fundamental operations
for discrete convex functions,
though not covering multimodular functions and discrete midpoint convex functions,
was conducted in \cite{MS01rel} 
at the early stage of discrete convex analysis.

Table~\ref{TBoperation3dcsetZ} 
is a summary of the behavior of discrete convex sets 
with respect to the operations of
direct sum, splitting, aggregation, and network transformation
discussed in this paper.
In the table, ``Y'' means that the set class is closed under the operation
and ``\textbf{\textit{N}}'' means it is not,
where we use different fonts for easier distinction.
For the results obtained in the paper,
specific references are made to the corresponding
propositions (Propositions \ref{PRsetdirsumMult}, \ref{PRsplitsetIC}, and \ref{PRsplitsetMult}) 
and counterexamples.
The results about M- and \Mnat-convex sets 
are not particularly new, as they are no more than restatements of well known facts 
in the literature of polymatroids and submodular functions
\cite{Fra11book,Fuj05book}.
These operations for jump systems are considered by 
Bouchet and Cunningham \cite{BouC95}
and Kabadi and Sridhar \cite{KS05jump}.
Table~\ref{TBoperation3dcfnZ}
offers a similar summary
for operations on functions,
with pointers to the major propositions 
(Propositions \ref{PRfndirsumMult}, \ref{PRsplitfnIC}, and \ref{PRsplitfnMult})
as well as to counterexamples in this paper.
Network induction for M-convex functions originates in \cite{Mstein96},
and that for jump M-convex functions is due to \cite{KMT07jump}.  
The reader is referred to
Tables 3 to 6 in \cite{Msurvop19}
for summaries 
about other operations such as restriction, projection, scaling, and convolution.

%%%%%%%%%% table %%%%%%%%%%
\begin{table}
\begin{center}
\caption{Operations on discrete convex sets}
\label{TBoperation3dcsetZ}

\medskip

%% To Typesetter: The Y and N in the table are intentionally staggered.  
%% Please do not align vertically, but use the macros \YES and \NO as they are.

%%\addtolength{\tabcolsep}{-3pt}
\begin{tabular}{l|c|c|c|c|l}
\hline  
 Discrete convex set & Direct & Splitting  & Aggrega-  & Network & Reference
\\ 
 & sum &  & tion & induction &
\\ \hline \hline
 Integer box  & \YES & \NO & \YES  & \NO   &
\\ 
%%%%%%%%
      &  & Ex.\ref{EXpointsplit} &  & Ex.\ref{EXpointsplit}  &   (this paper)
\\ \hline 
%%%%%%%%%%%%%
 Integrally convex & \YES & \YES & \NO  & \NO  
   &  
\\
%%%%%%%%
      &  & Prop.\ref{PRsplitsetIC} &  Ex.\ref{EXicvsetaggr} & Ex.\ref{EXicvsetaggr}  &  (this paper)
\\ \hline 
%%%%%%%%%%%%%
\Lnat-convex       & \YES & \NO & \NO & \NO  & 
\\ 
%%%%%%%%
      &  & Ex.\ref{EXpointsplit} & Ex.\ref{EXlnatsetaggr} & Ex.\ref{EXpointsplit}, \ref{EXlnatsetaggr}  &  (this paper)
\\ \hline 
%%%%%%%%%%%%%
L-convex           & \YES & \NO & \NO & \NO & 
\\ 
%%%%%%%%
      &  & Ex.\ref{EXlsetsplit} & Ex.\ref{EXlsetaggr} & Ex.\ref{EXlsetsplit}, \ref{EXlsetaggr}  &  (this paper)
\\ \hline 
%%%%%%%%%%%%%
\Mnat-convex       & \YES & \YES  & \YES  & \YES 
  & \cite{Fra11book,Mdcasiam,MS99gp} 
\\ 
%%%%%%%%
      &  &  &  &   &  
\\ \hline 
%%%%%%%%%%%%%
M-convex           & \YES & \YES & \YES  & \YES 
  & \cite{Fuj05book,Mdcasiam}
\\ 
%%%%%%%%
      &  &  &  &   &  
\\ \hline 
%%%%%%%%%%%%%
\hline
%%%%%%%%%%%%%%%%%%%%%%%%
Multimodular       & \YES & \YES & \NO  &  \NO   & 
\\ 
%%%%%%%%
      & Prop.\ref{PRsetdirsumMult} & Prop.\ref{PRsplitsetMult} & Ex.\ref{EXmmsetaggr} & Ex.\ref{EXmmsetaggr}  & (this paper)  
\\ \hline 
%%%%%%%%%%%%%
Disc.~midpt convex   & \NO & \NO & \NO & \NO  & 
\\ 
%%%%%%%%
      & Ex.\ref{EXdmcsetdirsum} & Ex.\ref{EXpointsplit} &  Ex.\ref{EXicvsetaggr} & Ex.\ref{EXicvsetaggr}  &  (this paper)
\\ \hline 
%%%%%%%%%%%%%
Simul.~exch. jump  & \YES & \YES & \YES & \YES   
  &  \cite{KS05jump,Mmnatjump19}  
\\ 
%%%%%%%%
      &  &  &  &   &  
\\ \hline 
%%%%%%%%%%%%%
Const-parity jump   & \YES & \YES & \YES & \YES   
  &  \cite{BouC95,KS05jump} 
\\ 
%%%%%%%%
      &  &  &  &   &  
\\ \hline 
%%%%%%%%%%%%%
\multicolumn{6}{l}{``Y'' means ``Yes, this set class is closed under this operation.''} 
\\
\multicolumn{6}{l}{``\textbf{\textit{N}}'' means ``No, this set class is not closed under this operation.''} 
\end{tabular}
%%\addtolength{\tabcolsep}{3pt}
\end{center}
\end{table}
%%%%%%%%%% table %%%%%%%%%%

%%%%%%%%%% table %%%%%%%%%%
\begin{table}
\begin{center}
\caption{Operations on discrete convex functions}
\label{TBoperation3dcfnZ}

\medskip

%% To Typesetter: The Y and N in the table are intentionally staggered.  
%% Please do not align vertically, but use the macros \YES and \NO as they are.

%%\addtolength{\tabcolsep}{-3pt}
\begin{tabular}{l|c|c|c|c|l}
\hline  
 Discrete  & Direct & Splitting  & Aggrega-  & Network & Reference
\\ 
 \ convex function & sum &  & tion & induction &
\\ \hline \hline
 Separable convex & \YES & \NO & \YES  & \NO   &
\\
%%%%%%%%
      &  & Ex.\ref{EXpointsplit} &  & Ex.\ref{EXpointsplit}  &   (this paper)
\\ \hline 
%%%%%%%%%%%%%
 Integrally convex & \YES & \YES & \NO  & \NO  &  
\\
%%%%%%%%
      &  & Prop.\ref{PRsplitfnIC} &  Ex.\ref{EXicvsetaggr} & Ex.\ref{EXicvsetaggr}  &  (this paper)
\\ \hline 
%%%%%%%%%%%%%
\Lnat-convex       & \YES & \NO & \NO & \NO  & 
\\
%%%%%%%%
      &  & Ex.\ref{EXpointsplit} & Ex.\ref{EXlnatsetaggr} & Ex.\ref{EXpointsplit}, \ref{EXlnatsetaggr}  &  (this paper)
\\ \hline 
%%%%%%%%%%%%%
L-convex           & \YES & \NO & \NO & \NO & 
\\
%%%%%%%%
      &  & Ex.\ref{EXlsetsplit} & Ex.\ref{EXlsetaggr} & Ex.\ref{EXlsetsplit}, \ref{EXlsetaggr}  &  (this paper)
\\ \hline 
%%%%%%%%%%%%%
\Mnat-convex       & \YES & \YES  & \YES  & \YES 
  & \cite{Mdcasiam} 
\\
%%%%%%%%
      &  &  &  &   &  
\\ \hline 
%%%%%%%%%%%%%
M-convex           & \YES & \YES & \YES  & \YES 
  & \cite{Mstein96,Mdcasiam} 
\\
%%%%%%%%
      &  &  &  &   &  
\\ \hline \hline
%%%%%%%%%%%%%
Multimodular       & \YES & \YES & \NO  &  \NO  & 
\\
%%%%%%%%
      & Prop.\ref{PRfndirsumMult} & Prop.\ref{PRsplitfnMult} & Ex.\ref{EXmmsetaggr} & Ex.\ref{EXmmsetaggr}  & (this paper)  
\\ \hline 
%%%%%%%%%%%%%
Globally d.m.c.  & \NO & \NO & \NO & \NO  & 
\\
%%%%%%%%
      & Ex.\ref{EXdmcsetdirsum}, \ref{EXdirsumdmcfn} & Ex.\ref{EXpointsplit} &  Ex.\ref{EXicvsetaggr} & Ex.\ref{EXicvsetaggr}  &  (this paper)
\\ \hline 
%%%%%%%%%%%%%
Locally d.m.c.  & \NO & \NO & \NO & \NO  & 
\\
%%%%%%%%
      & Ex.\ref{EXdmcsetdirsum}, \ref{EXdirsumdmcfn} & Ex.\ref{EXpointsplit} &  Ex.\ref{EXicvsetaggr} & Ex.\ref{EXicvsetaggr}  &  (this paper)
\\ \hline 
%%%%%%%%%%%%%
Jump \Mnat-convex  & \YES & \YES & \YES & \YES   
  &  \cite{Mmnatjump19}  
\\
%%%%%%%%
      &  &  &  &   &  
\\ \hline 
%%%%%%%%%%%%%
Jump M-convex   & \YES & \YES & \YES & \YES   
  &  \cite{KMT07jump}  
 \\ 
%%%%%%%%
      &  &  &  &   &  
\\ \hline 
%%%%%%%%%%%%%
\multicolumn{6}{l}{``Y'' means ``Yes, this function class is closed under this operation.''} 
\\
\multicolumn{6}{l}{``\textbf{\textit{N}}'' means ``No, this function class is not closed under this operation.''} 
\end{tabular}
%%\addtolength{\tabcolsep}{3pt}
\end{center}
\end{table}
%%%%%%%%%% table %%%%%%%%%%

This paper is organized as follows.
Section~\ref{SCfnclass} is a brief summary of the definitions of discrete convex
sets and functions,
including new observations
(Theorems \ref{THmmsetpolydes} and \ref{THmmfnargmin}, 
Example \ref{EXjumpMfnargmindim2}).
Section~\ref{SCsetope} treats operations on discrete convex sets
such as direct sum, splitting, aggregation, and network transformation.
Section~\ref{SCfnope} treats the corresponding operations on discrete convex functions.
Section~\ref{SCproof} gives the proofs.

%% 2019-12-03 / 2019-12-25 / 2020-02-10 / 2020-05-08

\section{Definitions of Discrete Convex Sets and Functions}
\label{SCfnclass}

In this section we provide a minimum account of 
definitions of discrete convex sets
$S \subseteq \ZZ\sp{n}$
and functions 
$f: \ZZ\sp{n} \to \RR \cup \{ +\infty \}$.
Let $N = \{ 1,2, \ldots, n \}$.

For $i \in \{ 1,2, \ldots, n \}$,
the $i$th unit vector is denoted by $\unitvec{i}$.
We define $\unitvec{0}=\veczero$
where
$\veczero =(0,0,\ldots,0)$. 
We also define $\vecone=(1,1,\ldots,1)$.

For a vector $x=(x_{1},x_{2}, \ldots, x_{n})$
 and a subset $A \subseteq \{ 1,2, \ldots, n \}$,
$x(A)$ denotes the component sum within $A$,
i.e.,
$x(A) = \sum \{ x_{i} \mid i \in A \}$.
The {\em positive} and {\em negative supports} of 
$x=(x_{1},x_{2}, \ldots, x_{n})$
are defined as
\begin{equation} \label{vecsupportdef}
 \suppp(x) = \{ i \mid x_{i} > 0 \},
\qquad 
 \suppm(x) = \{ i \mid x_{i} < 0 \}.
\end{equation}

The {\em indicator function} of a set 
$S \subseteq \ZZ\sp{n}$
is the function 
$\delta_{S}: \ZZ\sp{n} \to \{ 0, +\infty \}$
defined by
\begin{equation}  \label{indicatordef}
\delta_{S}(x)  =
   \left\{  \begin{array}{ll}
    0            &   (x \in S) ,      \\
   + \infty      &   (x \not\in S) . \\
                      \end{array}  \right.
\end{equation}
The convex hull of a set $S$ is denoted by $\overline{S}$.
The {\em effective domain} of a function $f$ means the set of $x$
with $f(x) <  +\infty$ and is denoted 
by $\dom f =   \{ x \in \ZZ\sp{n} \mid  f(x) < +\infty \}$.
We always assume that $\dom f$ is nonempty.

\subsection{Separable convexity}
\label{SCseparconvex}

For integer vectors 
$a \in (\ZZ \cup \{ -\infty \})\sp{n}$ and 
$b \in (\ZZ \cup \{ +\infty \})\sp{n}$ 
with $a \leq b$,
$[a,b]_{\ZZ}$ denotes the integer box  
(discrete rectangle, integer interval)
between $a$ and $b$.
A function
$f: \ZZ^{n} \to \RR \cup \{ +\infty \}$
in $x=(x_{1}, x_{2}, \ldots,x_{n}) \in \ZZ^{n}$
is called  {\em separable convex}
if it can be represented as
\begin{equation}  \label{sepfndef}
f(x) = \varphi_{1}(x_{1}) + \varphi_{2}(x_{2}) + \cdots + \varphi_{n}(x_{n})
\end{equation}
with univariate discrete convex functions
$\varphi_{i}: \ZZ \to \RR \cup \{ +\infty \}$, which, by definition, satisfy 
\begin{equation}  \label{univarconvdef}
\varphi_{i}(t-1) + \varphi_{i}(t+1) \geq 2 \varphi_{i}(t)
\qquad (t \in \ZZ).
\end{equation}

\subsection{Integral convexity}

For $x \in \RR^{n}$ the integral neighborhood of $x$ is defined
in \cite{FT90} as 
\begin{equation}  \label{intneighbordef}
N(x) = \{ z \in \ZZ^{n} \mid | x_{i} - z_{i} | < 1 \ (i=1,2,\ldots,n)  \}.
\end{equation}
It is noted that 
strict inequality ``$<$'' is used in this definition
and hence $N(x)$ admits an alternative expression
\begin{equation}  \label{intneighbordeffloorceil}
N(x) = \{ z \in \ZZ\sp{n} \mid
\lfloor x_{i} \rfloor \leq  z_{i} \leq \lceil x_{i} \rceil  \ \ (i=1,2,\ldots, n) \} ,
\end{equation}
where, for $t \in \RR$ in general, 
$\left\lceil  t   \right\rceil$ 
denotes the smallest integer not smaller than $t$
(rounding-up to the nearest integer)
and $\left\lfloor  t  \right\rfloor$
the largest integer not larger than $t$
(rounding-down to the nearest integer).
For a set $S \subseteq \ZZ^{n}$
and $x \in \RR^{n}$
we call the convex hull of $S \cap N(x)$ 
the {\em local convex hull} of $S$ at $x$.
A nonempty set $S \subseteq \ZZ^{n}$ is said to be 
{\em integrally convex} if
the union of the local convex hulls $\overline{S \cap N(x)}$ over $x \in \RR^{n}$ 
is convex \cite{Mdcasiam}.
This is equivalent to saying that,
for any $x \in \RR^{n}$, 
$x \in \overline{S} $ implies $x \in  \overline{S \cap N(x)}$.

It is recognized only recently that the concept of integrally convex sets is 
closely related (or essentially equivalent) to 
the concept of box-integer polyhedra.
Recall from \cite[Section~5.15]{Sch03} 
%%\cite[p.75]{Sch03}
that a polyhedron $P \subseteq \RR\sp{n}$ is called 
{\em box-integer}
if $P \cap \{ x \in \RR\sp{n} \mid a  \leq x \leq b \}$
is an integer polyhedron for each choice of integer vectors $a$ and $b$.
Then it is easy to see that 
if a set
$S \subseteq \ZZ^{n}$
is integrally convex, then its convex hull
$\overline{S}$ is a box-integer polyhedron, and conversely,
if $P$ is a box-integer polyhedron, then 
$S = P \cap \ZZ\sp{n}$ is an integrally convex set.

For a function
$f: \ZZ^{n} \to \RR \cup \{ +\infty  \}$
the {\em local convex extension} 
$\tilde{f}: \RR^{n} \to \RR \cup \{ +\infty \}$
of $f$ is defined 
as the union of all convex envelopes of $f$ on $N(x)$.  That is,
\begin{equation} \label{fnconvclosureloc2}
 \tilde f(x) = 
  \min\{ \sum_{y \in N(x)} \lambda_{y} f(y) \mid
      \sum_{y \in N(x)} \lambda_{y} y = x,  \ 
  (\lambda_{y})  \in \Lambda(x) \}
\quad (x \in \RR^{n}) ,
\end{equation} 
where $\Lambda(x)$ denotes the set of coefficients for convex combinations indexed by $N(x)$:
\[ 
  \Lambda(x) = \{ (\lambda_{y} \mid y \in N(x) ) \mid 
      \sum_{y \in N(x)} \lambda_{y} = 1, 
      \lambda_{y} \geq 0 \ \ \mbox{for all } \   y \in N(x)  \} .
\] 
If $\tilde f$ is convex on $\RR^{n}$,
then $f$ is said to be {\em integrally convex}
\cite{FT90}.
The effective domain of an integrally convex function is an integrally convex set.
A set $S \subseteq \ZZ\sp{n}$ is integrally convex if and only if its indicator function
$\delta_{S}: \ZZ\sp{n} \to \{ 0, +\infty \}$
is an integrally convex function.

Integral convexity of a function can be characterized as follows.

\begin{theorem}[{\cite[Theorem A.1]{MMTT20dmc}}]
\label{THmmtt19ThA1}
A function $f: \ZZ^{n} \to \RR \cup \{ +\infty  \}$
with $\dom f \not= \emptyset$
is integrally convex
if and only if, 
for every $x, y \in \ZZ\sp{n}$  we have \ 
\begin{equation*} 
%%\label{intcnvcondmid}
\tilde{f}\, \bigg(\frac{x + y}{2} \bigg) 
\leq \frac{1}{2} (f(x) + f(y)),
\end{equation*}
where $\tilde{f}$ is the local convex extension of $f$ 
defined by \eqref{fnconvclosureloc2}.
%%\vspace{-1.7\baselineskip} \\
%%\finbox
\end{theorem}

The reader is referred to 
\cite{MM19projcnvl,MMTT19proxIC,MT19subgrIC}
for recent developments
in the theory of integral convexity.

\subsection{L-convexity and discrete midpoint convexity}

\subsubsection{L-convex sets and functions}

A nonempty set $S \subseteq  \ZZ\sp{n}$ is called {\em \Lnat-convex} if
\begin{equation} \label{midptcnvset}
 x, y \in S
\ \Longrightarrow \
\left\lceil \frac{x+y}{2} \right\rceil ,
\left\lfloor \frac{x+y}{2} \right\rfloor  \in S ,
\end{equation}
where
$\left\lceil  z   \right\rceil
 = (\left\lceil  z_{1}   \right\rceil ,
   \left\lceil  z_{2}   \right\rceil ,
     \ldots, 
   \left\lceil  z_{n}   \right\rceil )$ 
and
$\left\lfloor  z   \right\rfloor
 = (\left\lfloor  z_{1}   \right\rfloor ,
   \left\lfloor  z_{2}   \right\rfloor ,
     \ldots, 
   \left\lfloor  z_{n}   \right\rfloor )$
for $z =(z_{1}, z_{2}, \ldots, z_{n}) \in \RR\sp{n}$. 
The property \eqref{midptcnvset} is called {\em discrete midpoint convexity}.

A function $f : \ZZ\sp{n} \to \RR \cup \{ +\infty \}$ with $\dom f \not= \emptyset$
is said to be {\em \Lnat-convex}
if it satisfies a quantitative version of discrete midpoint convexity,
i.e., if 
\begin{equation} \label{midptcnvfn}
 f(x) + f(y) \geq
   f \left(\left\lceil \frac{x+y}{2} \right\rceil\right) 
  + f \left(\left\lfloor \frac{x+y}{2} \right\rfloor\right) 
\end{equation}
holds for all $x, y \in \ZZ\sp{n}$.
The effective domain of an \Lnat-convex function is an \Lnat-convex set.
A set $S$ is \Lnat-convex if and only if its indicator function
$\delta_{S}$ is an \Lnat-convex function.
It is known \cite[Section 7.1]{Mdcasiam} that
\Lnat-convex functions can be characterized 
by several different conditions.

For example,
$f(x_{1}, x_{2},x_{3} ) = \max \{ x_{1}, x_{2}, x_{3} \}$
is an \Lnat-convex function.
Another function
$f(x_{1}, x_{2},x_{3} ) 
=  {x_{1}}\sp{2} + | x_{1} - x_{2} |  + ( x_{2} - x_{3} )\sp{2}$
is also \Lnat-convex.
More generally
\cite[Section~7.3]{Mdcasiam}, 
\begin{equation} \label{lfn2diffsepar}
 f(x)  =    \sum_{i = 1}\sp{n} \varphi_{i}(x_{i})
         + \sum_{i \not= j} \varphi_{ij}(x_{i}-x_{j}) 
\end{equation}
with univariate convex functions
$\varphi_{i}$ $(i = 1,2, \ldots, n)$
and $\varphi_{ij}$ $(i,j= 1,2, \ldots, n; i \not= j)$
is \Lnat-convex. 
A function of the form of \eqref{lfn2diffsepar}
is sometimes called a {\em 2-separable diff-convex function}.

A function 
$f(x_{1},x_{2}, \ldots, x_{n} )$
is said to be submodular
if
\begin{equation} \label{submfn}
f(x) + f(y) \geq f(x \vee y) + f(x \wedge y)
\end{equation}
holds for all $x, y \in \ZZ\sp{n}$, 
where
$x \vee y$ and $x \wedge y$ denote,
respectively, the vectors of componentwise maximum and minimum of $x$ and $y$, 
i.e.,
\begin{equation} \label{veewedgedef}
  (x \vee y)_{i} = \max(x_{i}, y_{i}),
\quad
  (x \wedge y)_{i} = \min(x_{i}, y_{i})
\qquad (i =1,2,\ldots, n).
\end{equation}

A function 
$f(x_{1},x_{2}, \ldots, x_{n} )$
with $\dom f \not= \emptyset$ is called {\em L-convex}
if it is submodular
 and there exists
$r \in \RR$ such that 
\begin{equation}\label{shiftlfnZ}
f(x + \vecone) = f(x) +  r
\end{equation}
for all $x  \in \ZZ\sp{n}$.
If $f$ is {\rm L}-convex,
the function
$g(x_{2}, \ldots, x_{n} ) := f(0, x_{2}, \ldots, x_{n} )$
is an \Lnat-convex function, and 
every \Lnat-convex function 
arises in this way.
For example,
$f(x_{1}, x_{2},x_{3} ) = \max \{ x_{1}, x_{2}, x_{3} \}$,
mentioned above as an \Lnat-convex function,
is actually L-convex.
A 2-separable diff-convex function in \eqref{lfn2diffsepar} is L-convex if 
$\varphi_{i}=0$ for $i = 1,2, \ldots, n$.

A nonempty set $S$ is called {\em L-convex} if its indicator function $\delta_{S}$
is an L-convex function.
The effective domain of an L-convex function is an L-convex set.

\subsubsection{Discrete midpoint convex sets and functions}

A nonempty set $S \subseteq  \ZZ\sp{n}$ is said to be {\em discrete midpoint convex} 
\cite{MMTT20dmc} if
\begin{equation} \label{dirintcnvsetdef}
 x, y \in S, \ \| x - y \|_{\infty} \geq 2
\ \Longrightarrow \
\left\lceil \frac{x+y}{2} \right\rceil ,
\left\lfloor \frac{x+y}{2} \right\rfloor  \in S.
\end{equation}
This condition is  weaker than the defining condition (\ref{midptcnvset})
for an \Lnat-convex set,
and hence every \Lnat-convex set
is a discrete midpoint convex set.

A function $f: \ZZ\sp{n} \to \RR \cup \{ +\infty \}$
with $\dom f \not= \emptyset$
is called {\em  globally discrete midpoint convex} if
the discrete midpoint convexity 
\eqref{midptcnvfn}
is satisfied by every pair $(x, y) \in \ZZ\sp{n} \times \ZZ\sp{n}$
with $\| x - y \|_{\infty} \geq 2$.
The effective domain of a
globally discrete midpoint convex function
is necessarily a discrete midpoint convex set.
A function $f: \ZZ\sp{n} \to \RR \cup \{ +\infty \}$
with $\dom f \not= \emptyset$
is called {\em  locally discrete midpoint convex}
if $\dom f$ is a discrete midpoint convex set
and the discrete midpoint convexity (\ref{midptcnvfn})
is satisfied by every pair $(x, y) \in \ZZ\sp{n} \times \ZZ\sp{n}$
with $\| x - y \|_{\infty} = 2$ (exactly equal to $2$).
Obviously, 
every \Lnat-convex function is
globally discrete midpoint convex, and 
every globally discrete midpoint convex function is
locally discrete midpoint convex.
We sometimes abbreviate ``discrete midpoint convex(ity)'' to ``d.m.c.''

The inclusion relations for sets and functions 
equipped with (variants of) L-convexity are summarized as follows:
\begin{align*}
&
\{ \mbox{\rm {\rm L}-convex sets} \} \subsetneqq \ 
\{ \mbox{\rm \Lnat-convex sets } \} \subsetneqq \ 
\{ \mbox{\rm discrete midpoint convex sets} \},
%%\label{setfamdmc}
\\ &
\{ \mbox{\rm {\rm L}-convex fns} \} \subsetneqq \ 
\{ \mbox{\rm \Lnat-convex fns} \} \subsetneqq \ 
\{ \mbox{\rm globally d.m.c. fns} \} \subsetneqq \ 
\{ \mbox{\rm locally d.m.c. fns} \}.
%%\label{fnfamdmc}
\end{align*}

\subsection{M-convexity and jump M-convexity}

\subsubsection{M-convex sets and functions}
\label{SCmnatfn}

A nonempty set $S \subseteq \ZZ\sp{n}$ 
is called an {\em \Mnat-convex set}
if it satisfies the following exchange property:
\begin{description}
\item[\Bnvexb] 
For any $x, y \in S$ and $i \in \suppp(x-y)$, we have
(i)
$x -\unitvec{i} \in S$ and $y+\unitvec{i} \in S$
\ or \  
\\
(ii) there exists some $j \in \suppm(x-y)$ such that
$x-\unitvec{i}+\unitvec{j}  \in S$ and $y+\unitvec{i}-\unitvec{j} \in S$.
\end{description}
\Mnat-convex set is an alias for the set of integer points in an integral generalized polymatroid.
In particular, the family of independent sets of a matroid
can be regarded as an \Mnat-convex set consisting of $\{ 0, 1\}$-vectors.

A function
$f: \ZZ\sp{n} \to \RR \cup \{ +\infty \}$
with $\dom f \not= \emptyset$
is called {\em \Mnat-convex}, if,
for any $x, y \in \dom f$ and $i \in \suppp(x-y)$, 
we have (i)
\begin{equation}  \label{mconvex1Z}
f(x) + f(y)  \geq  f(x -\unitvec{i}) + f(y+\unitvec{i})
\end{equation}
or (ii) there exists some $j \in \suppm(x-y)$ such that
\begin{equation}  \label{mconvex2Z}
f(x) + f(y)   \geq 
 f(x-\unitvec{i}+\unitvec{j}) + f(y+\unitvec{i}-\unitvec{j}) .
\end{equation}
This property is referred to as the {\em exchange property}.
A more compact expression of this exchange property is as follows:
\begin{description}
\item[\Mnvexb]
 For any $x, y \in \dom f$ and $i \in \suppp(x-y)$, we have
\begin{equation} \label{mnconvexc2Z}
f(x) + f(y)   \geq 
\min_{j \in \suppm(x - y) \cup \{ 0 \}} 
 \{ f(x - \unitvec{i} + \unitvec{j}) + f(y + \unitvec{i} - \unitvec{j}) \},
\end{equation}
%% 2019-08-24 correction
%% 2019-12-03 revision
\end{description}
where $\unitvec{0}=\veczero$ (zero vector).

For example,
$f(x_{1}, x_{2},x_{3} ) = | x_{1} + x_{2} + x_{3} | 
+ ( x_{1} + x_{2} )\sp{2} + {x_{3}}\sp{2}$
is an \Mnat-convex function.
More generally
\cite[Section~6.3]{Mdcasiam}, 
a laminar convex function is \Mnat-convex,
where a function $f$ is called {\em laminar convex}
if it can be represented as
\begin{equation} \label{mnatfnlaminar}
 f(x)  =  \sum_{A \in \mathcal{T}} \varphi_{A}(x(A))
\end{equation}
for a laminar family $\mathcal{T} \subseteq 2\sp{N}$
(i.e., $A \cap B = \emptyset$, $ A \subseteq B$, or $ A \supseteq B$
for any $A, B \in \mathcal{T}$)
and a family of 
univariate discrete convex functions 
$\varphi_{A}: \ZZ \to \RR \cup \{ +\infty \}$ 
indexed by $A \in \mathcal{T}$.

\Mnat-convex functions can be characterized by 
a number of different exchange properties including a local exchange property
under the assumption that 
function $f$ is (effectively) defined on an \Mnat-convex set.
See \cite{MS18mnataxiom} 
as well as \cite[Theorem~4.2]{Mdcaeco16} and \cite[Theorem~6.8]{ST15jorsj}.

If a set $S \subseteq \ZZ\sp{n}$ lies on a hyperplane with a constant component sum
(i.e., $x(N) = y(N)$ for all  $x, y \in S$),
the exchange property \Bnvex
takes a simpler form
(without the possibility of the first case (i)): 
\begin{description}
\item[\Bvexb] 
For any $x, y \in S$ and $i \in \suppp(x-y)$, 
there exists some $j \in \suppm(x-y)$ such that
$x-\unitvec{i}+\unitvec{j}  \in S$ and $y+\unitvec{i}-\unitvec{j} \in S$.
\end{description}
A nonempty set $S \subseteq \ZZ\sp{n}$ having this exchange property 
is called an {\em M-convex set},
which is an alias for the set of integer points in an integral base polyhedron.  
In particular, the basis family of a matroid
can be identified precisely with an M-convex set consisting of $\{ 0, 1\}$-vectors.

An \Mnat-convex function 
whose effective domain is an M-convex set
is called an {\em M-convex function}
\cite{Mstein96,Mdca98,Mdcasiam}.
In other words,  a function 
$f: \ZZ\sp{n} \to \RR \cup \{ +\infty \}$
with $\dom f \not= \emptyset$
is M-convex
if and only if it satisfies
the exchange property:
\begin{description}
\item[\Mvexb]
 For any $x, y \in \dom f$  and $i \in \suppp(x-y)$, 
there exists
$j \in \suppm(x-y)$ such that
\eqref{mconvex2Z} holds.
%% 2019-09-18 %%%
\end{description}
M-convex functions can be characterized by a local exchange property
under the assumption 
that function $f$ is (effectively) defined on an M-convex set.
See \cite[Section 6.2]{Mdcasiam}.

M-convex functions and
\Mnat-convex functions are equivalent concepts,
in that \Mnat-convex functions in $n$ variables 
can be obtained as projections of M-convex functions in $n+1$ variables.
More formally, 
%%let ``$0$'' denote a new element not in $N$ and
%%$\tilde{N} = \{0\} \cup N = \{ 0,1,\ldots,n \}$.
a function $f : \ZZ\sp{n} \to \RR \cup \{ +\infty \}$
is \Mnat-convex
if and only if  the function 
$\tilde{f} : \ZZ\sp{n+1} \to \RR \cup \{ +\infty \}$ 
defined by
\begin{equation} \label{mfnmnatfnrelationvex}
  \tilde{f}(x_{0},x) = \left\{ \begin{array}{ll}
      f(x)    & \mbox{ if $x_{0} = {-}x(N)$} \\
      +\infty & \mbox{ otherwise}
    \end{array}\right.
 \qquad ( x_{0} \in \ZZ, x \in \ZZ\sp{n})
\end{equation}
is an M-convex function.

\subsubsection{Jump systems and jump M-convex functions}
\label{SCjumpMdef}

Let $x$ and $y$ be integer vectors.
The smallest integer box containing $x$ and $y$
is given by $[x \wedge y,x \vee y]_{\ZZ}$.
A vector $s \in \ZZ\sp{n}$ is called an {\em $(x,y)$-increment} if
$s= \unitvec{i}$ or $s= -\unitvec{i}$ for some $i \in N$ and 
$x+s \in [x \wedge y,x \vee y]_{\ZZ}$.

A nonempty set $S \subseteq \ZZ\sp{n}$ is said 
to be a {\em jump system} 
\cite{BouC95}
if satisfies an exchange axiom,
called the {\em 2-step axiom}:
\begin{description}
\item[(2-step axiom)]
For any  $x, y \in S$
and any $(x,y)$-increment $s$ with $x+s \not\in S$,
there exists an $(x+s,y)$-increment $t$ such that $x+s+t \in S$.
\end{description}
Note that we have the possibility of $s = t$ in the 2-step axiom.

A set $S \subseteq \ZZ\sp{n}$ is called a {\em constant-sum system}
if $x(N)=y(N)$ for any $x,y \in S$.
A constant-sum jump system is nothing but an M-convex set.

A set $S \subseteq \ZZ\sp{n}$ is called a {\em constant-parity system}
if $x(N)-y(N)$ is even for any $x,y \in S$.
It is known \cite{Mmjump06} that a 
{\em constant-parity jump system}
(or {\em c.p.~jump system})
is characterized  by
\begin{description}
\item[(J-EXC)]
For any  $x, y \in S$ and any $(x,y)$-increment $s$,
there exists an $(x+s,y)$-increment $t$ such that
$x+s+t \in S$ and $y-s-t \in S$.
\end{description}

A function  $f: \ZZ\sp{n} \to \RR \cup \{ +\infty \}$
with $\dom f \not= \emptyset$
is called%
\footnote{%%%%%%%%%%%%%%%%%%%%%%%%%
This concept (``jump M-convex function'')  is the same as ``M-convex function on a jump system"
in \cite{KMT07jump,Mmjump06}.
} %%%% footnote %%%%%%%%%%%%%%%%%%%
{\em jump M-convex} 
if it satisfies
the following exchange axiom:
\begin{description}
\item[(JM-EXC)]
For any  $x, y \in \dom f$
and any $(x,y)$-increment $s$,
 there exists an $(x+s,y)$-increment $t$ such that
$x+s+t \in \dom f$,  \  $y-s-t \in \dom f$, and
\begin{equation}  \label{jumpmexc2}
 f(x)+f(y) \geq  f(x+s+t)  + f(y-s-t) .
\end{equation} 
\end{description}
%% To Typesetter: Please do not indent here %%
\noindent
The effective domain of a jump M-convex function is 
a constant-parity jump system.

A jump system 
is called a {\em simultaneous exchange jump system}
(or {\em s.e.~jump system}) \cite{Mmnatjump19}
if it satisfies
the following exchange axiom
\begin{description}
\item[(J$\sp{\natural}$-EXC)]
For any  $x, y \in S$ and any $(x,y)$-increment $s$, we have
(i) $x+s \in S$ and $y-s \in S$, or 
(ii)
there exists an $(x+s,y)$-increment $t$ such that
$x+s+t \in S$ and $y-s-t \in S$.
\end{description}
%% To Typesetter: Please do not indent here %%
\noindent
Every constant-parity jump system
is a simultaneous exchange jump system,
since the condition (J-EXC) implies (J$\sp{\natural}$-EXC).
Not every jump system is a simultaneous exchange jump system,
as is shown in \cite[Examples 2.2 and 2.3]{Msurvop19}.

A function  $f: \ZZ\sp{n} \to \RR \cup \{ +\infty \}$
with $\dom f \not= \emptyset$
is called 
{\em jump \Mnat-convex}
\cite{Mmnatjump19}
if it satisfies
the following exchange axiom
\begin{description}
\item[(J\Mnat-EXC)]
For any  $x, y \in \dom f$
and any $(x,y)$-increment $s$, we have
\\
(i) $x+s \in \dom f$, \  $y-s \in \dom f$, and 
\begin{equation}  \label{jumpmexc1}
 f(x)+f(y) \geq  f(x+s)  + f(y-s) , 
\end{equation}
or (ii) there exists an $(x+s,y)$-increment $t$ such that
$x+s+t \in \dom f$,  \  $y-s-t \in \dom f$, and
\eqref{jumpmexc2} holds.
\end{description}

The condition 
(J\Mnat-EXC) is weaker than (JM-EXC), and hence
every jump M-convex function is a jump \Mnat-convex function.
However, the concepts of jump M-convexity and
jump \Mnat-convexity are in fact equivalent to each other in the sense that 
jump \Mnat-convex functions in $n$ variables 
can be identified with jump M-convex functions in $n+1$ variables.
More specifically, for any integer vector $x \in \ZZ\sp{n}$
we define $\pi(x)=0$ if
its component sum
$x(N)$ is even, and
$\pi(x)=1$ if $x(N)$ is odd.
%%and let $\tilde{N} = \{0\} \cup N = \{ 0,1,\ldots,n \}$
%%with a new element ``$0$'' not in $N$.
It is known \cite{Mmnatjump19} that 
a function $f : \ZZ\sp{n} \to \RR \cup \{ +\infty \}$ is 
jump \Mnat-convex
if and only if  the function 
$\tilde{f} : \ZZ\sp{n+1} \to \RR \cup \{ +\infty \}$ 
defined by
\begin{equation}  \label{ftildeJ}
\tilde f(x_{0}, x)= 
   \left\{  \begin{array}{ll}
   f(x)            &   (x_{0} = \pi(x))       \\
   + \infty      &   (\mbox{\rm otherwise})  \\
                      \end{array}  \right.
 \qquad ( x_{0} \in \ZZ, x \in \ZZ\sp{n})
\end{equation}
is a jump M-convex function.

The inclusion relations for sets and functions 
equipped with (variants of) M-convexity 
is summarized as follows:
\begin{align*}
& 
\{ \mbox{\rm M-convex sets} \}  \subsetneqq \
  \left\{  \begin{array}{l}  \{ \mbox{\rm \Mnat-convex  sets} \}
                          \\ \{ \mbox{\rm c.p. jump systems} \} \end{array}  \right \}
\subsetneqq \   \{ \mbox{\rm s.e. jump systems} \} 
\subsetneqq \   \{ \mbox{\rm jump systems} \} ,
%% \label{setfamjump0}
\\
&
\{ \mbox{\rm M-convex fns} \}  \subsetneqq \
  \left\{  \begin{array}{l}  \{ \mbox{\rm \Mnat-convex fns} \}
                          \\ \{ \mbox{\rm jump M-convex fns} \} \end{array}  \right \}
\subsetneqq \  
\{ \mbox{\rm jump \Mnat-convex fns} \} .
%%\label{fnfamjump0}
\end{align*}
It is noted that no convexity class is introduced for functions defined on general jump systems.

Finally we mention an example to show that a jump M-convex function
may not look like a convex function in the intuitive sense.
Nevertheless, jump M- and \Mnat-convex functions find applications in several fields including
matching theory \cite{BK12,KST12cunconj,KT09evenf,Tak14fores}
and algebra \cite{Bra10halfplane}.

\begin{example}  \rm  \label{EXjumpMfndim2A}
Let $S$ be a subset of $\ZZ\sp{2}$ 
defined by
\[
S = \{ (x_1, x_2) \in \ZZ\sp{2} 
\mid 
 0 \leq x_1 \leq 3, \ 0 \leq x_2 \leq 3, \
 \mbox{$x_1 + x_2$: even}
  \}.
\]
This set is a constant-parity jump system.
Consider 
$f: S \to \RR$ 
defined by
\begin{equation}  \label{jumpMfndim2A}
 f(x_1,x_2) = \begin{cases}
 0 &  (\  x_1, x_2 \in \{ 0,2 \} \ ),
\\
 1 &  (\  x_1, x_2 \in \{ 1,3 \} \ ),
               \end{cases}
\end{equation}
which may be shown as
\[
f(x_{1},x_{2}) = 
\begin{array}{|rrrr|}
 \hline 
 -  & 1 & - & 1  
\\
   0 & - & 0 & - 
\\
 -  & 1 & - & 1  
\\
   0 & - & 0 & - 
\\ \hline 
\end{array}  \ .
\]
This function is jump M-convex.
Indeed, for 
$x=(0,0)$, $y=(2,2)$, and $s=(1,0)$, for example,
we can take $t=(1,0)$, for which  
$x+s+t=(2,0)$, $y-s-t=(0,2)$, and
$f(x)+f(y) = 0 + 0 = f(x+s+t)  + f(y-s-t)$ in \eqref{jumpmexc2}.
For 
$x=(0,0)$, $y=(3,3)$, and $s=(1,0)$, 
we can choose $t=(1,0)$ or $t=(0,1)$.
For either choice we have
$f(x+s+t)  + f(y-s-t) = 1 = f(x)+f(y)$.

It is noted that the function $f$ above arises from
the degree sequences of a graph as in \cite[Example 2.2]{Mmjump06}.
Let $G=(V,E)$ be an undirected graph with vertex set
$V= \{ v_{1}, v_{2} \}$ and
edge set $E$ consisting of three edges,
$E= \{ (v_{1}, v_{2}),  (v_{1}, v_{1}),  (v_{2}, v_{2})  \}$,
where $(v_{i}, v_{i})$ denotes a self-loop at $v_{i}$ for $i=1,2$.
The set $S$ above is the degree system (the set of the degree sequences of a subgraph)
of this graph $G$, and $f(x)$
coincides with the minimum weight of a subgraph with degree sequence $x$
when the $(v_{1}, v_{2})$ has weight 1 and the self-loops have weight 0.
%% \finbox
\end{example}

\subsection{Multimodularity}
\label{SCmultimod}

Recall that $\unitvec{i}$ denotes the $i$th unit vector for $i =1,2,\ldots, n$,
and 
$\mathcal{F} \subseteq \ZZ\sp{n}$ 
be the set of vectors defined by
\begin{equation} \label{multimodirection1}
\mathcal{F} = \{ -\unitvec{1}, \unitvec{1}-\unitvec{2}, \unitvec{2}-\unitvec{3}, \ldots, 
  \unitvec{n-1}-\unitvec{n}, \unitvec{n} \} .
\end{equation}
A finite-valued function $f: \ZZ\sp{n} \to \RR$
is said to be {\em multimodular}
\cite{Haj85}
if it satisfies
\begin{equation} \label{multimodulardef1}
 f(z+d) + f(z+d') \geq   f(z) + f(z+d+d')
\end{equation}
for all 
$z \in \ZZ\sp{n}$ and all distinct $d, d' \in \mathcal{F}$.
It is known \cite[Proposition 2.2]{Haj85} 
that $f: \ZZ\sp{n} \to \RR$
is multimodular if and only if the function 
$\tilde f: \ZZ\sp{n+1} \to \RR$ 
defined by 
\begin{equation} \label{multimodular1}
 \tilde f(x_{0}, x) = f(x_{1}-x_{0},  x_{2}-x_{1}, \ldots, x_{n}-x_{n-1})  
 \qquad ( x_{0} \in \ZZ, x \in \ZZ\sp{n})
\end{equation}
is submodular in $n+1$ variables. 
This characterization enables us to define multimodularity 
for a function that may take the infinite value $+\infty$.
That is, we say
\cite{MM19multm,Mmult05} 
that a function $f: \ZZ\sp{n} \to \RR \cup \{ +\infty \}$ with $\dom f \not= \emptyset$
is multimodular if the function 
$\tilde f: \ZZ\sp{n+1} \to \RR \cup \{ +\infty \}$
associated with $f$ by (\ref{multimodular1}) is submodular.

Multimodularity and \Lnat-convexity
have the following close relationship.

\begin{theorem}[\cite{Mmult05}]  \label{THmmfnlnatfn}
A function $f: \ZZ\sp{n} \to \RR \cup \{ +\infty \}$
is multimodular if and only if the function $g: \ZZ\sp{n} \to \RR \cup \{ +\infty \}$
defined by
\begin{equation} \label{mmfnGbyF}
 g(p) = f(p_{1}, \  p_{2}-p_{1}, \  p_{3}-p_{2}, \ldots, p_{n}-p_{n-1})  
 \qquad ( p \in \ZZ\sp{n})
\end{equation}
is \Lnat-convex.
%% \finbox
\end{theorem}

Note that the relation (\ref{mmfnGbyF}) between $f$ and $g$ can be rewritten as
\begin{equation} \label{mmfnFbyG}
 f(x) = g(x_{1}, \  x_{1}+x_{2}, \  x_{1}+x_{2}+x_{3}, 
   \ldots, x_{1}+ \cdots + x_{n})  
 \qquad ( x \in \ZZ\sp{n}) .
\end{equation}
Using a bidiagonal matrix 
$D=(d_{ij} \mid 1 \leq i,j \leq n)$ defined by
\begin{equation} \label{matDdef}
 d_{ii}=1 \quad (i=1,2,\ldots,n),
\qquad
 d_{i+1,i}=-1 \quad (i=1,2,\ldots,n-1),
\end{equation}
we can express (\ref{mmfnGbyF}) and (\ref{mmfnFbyG}) 
more compactly  as $g(p)=f(Dp)$ and $f(x)=g(D\sp{-1}x)$, respectively. 
The matrix $D$ is unimodular, and its inverse $D\sp{-1}$ 
%%$(D\sp{-1})_{ij}=1$ for $i \geq j$ and $(D\sp{-1})_{ij}=0$ for $i < j$.
is a lower triangular integer matrix 
whose $(i,j)$ entry is given by
\begin{equation} \label{matDinv}
(D\sp{-1})_{ij}=
 \begin{cases}
 1 & (i \geq j), \\
 0 &  (i < j).
 \end{cases}
\end{equation}
For $n=5$, for example, we have
\begin{equation} \label{matDdim5}
D = 
\left[ \begin{array}{rrrrr}
1 & 0 & 0 & 0 & 0 \\
-1 & 1 & 0 & 0 & 0 \\
0 & -1 & 1 & 0 & 0\\
0 & 0 & -1 & 1 & 0 \\
0 & 0 & 0 & -1 & 1 \\
\end{array}\right],
\qquad
D\sp{-1} = 
\left[ \begin{array}{rrrrr}
1 & 0 & 0 & 0 & 0 \\
1 & 1 & 0 & 0 & 0 \\
1 & 1 & 1 & 0 & 0 \\
1 & 1 & 1 & 1 & 0 \\
1 & 1 & 1 & 1 & 1 \\
\end{array}\right].
\end{equation}

A nonempty set $S$ is called {\em multimodular}
if its indicator function $\delta_{S}$ is multimodular. 
A multimodular set $S$ can be represented as 
$S = \{ x = D p \mid  p \in T \}$
for some \Lnat-convex set $T$,
where $T$ is uniquely determined from $S$ as 
$T = \{ p=D\sp{-1} x \mid  x \in S \}$.
It follows from (\ref{mmfnGbyF}) that
the effective domain of a multimodular function is a multimodular set.

A polyhedral description of a multimodular set is given as follows.
A subset of the index set $N = \{ 1,2,\ldots, n \}$
is said to be {\em consecutive}
if it consists of consecutive numbers, that is,
it is a set of the form 
$\{ k, k+1, \ldots, l -1, l \}$
for some $k \leq l$.

\begin{theorem} \label{THmmsetpolydes}
A set $S \subseteq \ZZ\sp{n}$ is multimodular if and only if
\begin{align*}
 S &= \{ x \in \ZZ\sp{n} \mid  a_{I} \leq x(I) \leq b_{I}
      \ \ (\mbox{\rm $I$: consecutive interval in $N$})   \}
\end{align*}
for some 
$a_{I} \in \ZZ \cup \{ -\infty \}$ and $b_{I} \in \ZZ \cup \{ +\infty \}$
indexed by consecutive intervals $I \subseteq N$.
\end{theorem}
\begin{proof}
As is well known (\cite[Section 5.5]{Mdcasiam}),
an \Lnat-convex set can be described by a system of inequalities of the form
$p_{i} - p_{j} \leq d_{ij}$ and  $a_{i} \leq p_{i} \leq b_{i}$.
On substituting
$p_{i} =  x_{1}+x_{2}+ \cdots + x_{i}$
$(i=1,2,\ldots,n)$
into these inequalities,
we obtain the claim.
\end{proof}

\subsection{Discrete convexity of functions in terms of the minimizers}

In this section we discuss how discrete convexity of functions
can be characterized in terms of the discrete convexity of the minimizer sets.

For a function 
$f: \ZZ\sp{n} \to \RR \cup \{ +\infty \}$ 
and a vector $c \in \RR\sp{n}$, 
$f[-c]$ will denote the function defined by
\[
  f[-c](x) = f(x) - \sum_{i=1}\sp{n} c_{i} x_{i}
\qquad (x \in \ZZ\sp{n}).
\]
It is often the case that $f$
is equipped with some kind of discrete convexity
if and only if, for every $c \in \RR\sp{n}$,
the set of the minimizers of $f[-c]$, i.e., 
\[
\argmin f[-c] = \{ x \in \ZZ\sp{n} \mid f[-c](x) \leq f[-c](y) \mbox{ for all $y \in \ZZ\sp{n}$} \}
\]
is equipped with the discrete convexity of the same kind.
This implies that the concept of 
discrete convex functions can also be defined from that of
discrete convex sets.

Indeed the following facts are known.

\begin{theorem}  \label{THfnargminG}
Let 
$f: \ZZ\sp{n} \to \RR \cup \{ +\infty \}$ be a function 
that is convex-extensible or has a bounded nonempty effective domain.%
\footnote{%%%%%%
In Part (4) for an L-convex function $f$, 
the boundedness of the effective domain
is to be understood as the boundedness of 
$\dom f$ intersected with a coordinate plane  $\{ x \mid x_{i} = 0 \}$
for some (or any) $i \in N$.
Note that the effective domain of an L-convex function 
has the invariance in the direction of $\vecone$. 
}%%%% footnote %%%%%%%%%%%%

%% Revision/Correction 2020-02-08

%% To Typesetter:  Please do NOT change the labels "(1)", "(2)", ....  below, 
%% as they are cited at other places.

\noindent
{\rm (1)} 
$f$ is separable convex
if and only if
$\argmin f[-c]$ is an integer box
for each $c \in \RR\sp{n}$.

\noindent
{\rm (2)} 
$f$ is integrally convex
if and only if
$\argmin f[-c]$ is an integrally convex set
for each $c \in \RR\sp{n}$.

\noindent
{\rm (3)} 
$f$ is \Lnat-convex
if and only if
$\argmin f[-c]$ is an \Lnat-convex set
for each $c \in \RR\sp{n}$.

\noindent
{\rm (4)} 
$f$ is L-convex
if and only if
$\argmin f[-c]$ is an L-convex set
for each $c \in \RR\sp{n}$.

\noindent
{\rm (5)} 
$f$ is \Mnat-convex
if and only if
$\argmin f[-c]$ is an \Mnat-convex set
for each $c \in \RR\sp{n}$.

\noindent
{\rm (6)} 
$f$ is M-convex
if and only if
$\argmin f[-c]$ is an M-convex set
for each $c \in \RR\sp{n}$.
\end{theorem}
\begin{proof}
It follows easily from the definitions that
the only-if parts in all cases (1)--(6) 
hold without the assumption of convex-extensibility or boundedness of $\dom f$.

%% Revision/Correction 2020-02-11
The if-parts under the assumption of bounded $\dom f$ are known in the literature.
Part (1) for separable convexity is obvious.
Part (2) for integral convexity is given in  
\cite[Theorem 3.29]{Mdcasiam}.
Parts (3) and (4) for \Lnat- and L-convexity are given in  
\cite[Theorem 7.17]{Mdcasiam}.
Parts (5) and (6) for \Mnat- and M-convexity are given in  
\cite[Theorem 6.30]{Mdcasiam}.

%% Revision/Correction 2020-02-11
The proof of the if-part
under the assumption of convex-extensibility of  $f$
can be reduced to the case of a bounded effective domain.
We demonstrate this reduction for \Lnat-convex functions.
A function $f$ is \Lnat-convex 
if and only if its restriction to every finite box is \Lnat-convex.
Let $f_{[a,b]}$ denote the restriction of $f$ 
to a finite integer box $[a,b]=[a,b]_{\ZZ}$,
and note that
$(f[-c])_{[a,b]} =   f_{[a,b]}[-c]$.
When $f$ is convex-extensible,  we have the relation
\begin{equation}  \label{argminfcabfabc}
 (\argmin f[-c]) \cap [a,b] =   \argmin (f_{[a,b]}[-c]) .
\end{equation}
By the assumption, $\argmin f[-c]$ is an \Lnat-convex set,
from which follows that
its intersection with the box $[a,b]$,
i.e., 
$(\argmin f[-c]) \cap [a,b]$,
is also an \Lnat-convex set.
Hence, by \eqref{argminfcabfabc},
$\argmin (f_{[a,b]}[-c])$ is an \Lnat-convex set for every $c$.
Since $\dom f_{[a,b]}$ is bounded, 
$f_{[a,b]}$ is an \Lnat-convex function.
Therefore, $f$ is \Lnat-convex.
The same argument is valid for other kinds of discrete convex functions.
\end{proof}

Moreover, we can show a similar statement for multimodularity,
which does not seem to have been made in the literature.

\begin{theorem}  \label{THmmfnargmin}
%% Revision/Correction 2020-02-11
Let $f: \ZZ\sp{n} \to \RR \cup \{ +\infty \}$ be a function 
that is convex-extensible or has a bounded nonempty effective domain.
Then $f$ is multimodular 
if and only if
$\argmin f[-c]$ is a multimodular set
for each $c \in \RR\sp{n}$.
\end{theorem}
\begin{proof}
This is a straightforward translation of 
Theorem~\ref{THfnargminG} (3) for an \Lnat-convex function.
%% \cite[Theorem 7.17]{Mdcasiam} 
Let $g(p)=f(Dp)$.
By Theorem~\ref{THmmfnlnatfn},
$f$ is multimodular if and only if $g$ is \Lnat-convex,
whereas the relation
\[
\argmin f[-c] = \{ x = Dp \mid p \in \argmin g[-D\sp{\top} c]  \ \}
\]
shows that $\argmin f[-c]$ is multimodular if and only if
$\argmin g[-c']$ is \Lnat-convex for $c' = D\sp{\top} c$.
\end{proof}

Jump M-convexity as well as jump \Mnat-convexity does not admit such characterization.
This is demonstrated by the following example.

\begin{example}  \rm  \label{EXjumpMfnargmindim2}
Let $S$ be a subset of $\ZZ\sp{2}$ 
defined by
\[
S = \{ (x_1, x_2) \in \ZZ\sp{2} 
\mid 
 0 \leq x_1 \leq 4, \ 0 \leq x_2 \leq 4, \
 \mbox{$x_1 + x_2$: even}
  \}.
\]
This set is a constant-parity jump system.
Consider 
$f: S \to \RR$ 
defined by
\begin{equation}  \label{jumpMfnargmindim2}
f(x_1,x_2) = \begin{cases}
 0 &  (x_1, x_2 \in \{ 0,2,4 \}),
\\
 \alpha  &  ((x_1, x_2) = (1,1), (1,3)),
\\
 \beta  &  ((x_1, x_2) = (3,1), (3,3))
               \end{cases}
\end{equation}
with parameters $\alpha$ and $\beta$, which may be shown as
\[
f(x_{1},x_{2}) =
\begin{array}{|ccccc|}
 \hline 
   0 & - & 0 & - & 0 
\\
 -  & \alpha & - & \beta  & -
\\
   0 & - & 0 & - & 0 
\\
 -  & \alpha & - & \beta  & -
\\
   0 & - & 0 & - & 0 
\\ \hline 
\end{array}  \ .
\]
This function is jump M-convex if and only if
$\alpha = \beta$.
Indeed, for 
$x=(0,0)$, $y=(2,2)$, and $s=(1,0)$, for example,
we can take $t=(1,0)$, for which  
$x+s+t=(2,0)$, $y-s-t=(0,2)$, and
$f(x)+f(y) = 0 + 0 = f(x+s+t)  + f(y-s-t)$ in 
\eqref{jumpmexc2}.
For 
$x=(0,0)$, $y=(3,3)$, and $s=(1,0)$, 
we can choose $t=(1,0)$ or $t=(0,1)$.
For either choice we have
$f(x+s+t)  + f(y-s-t) = \alpha$,
while $f(x)+f(y) = \beta$.
Therefore, the inequality \eqref{jumpmexc2} is satisfied if and only if 
$\alpha \leq \beta$.
By considering 
$x=(4,4)$, $y=(1,1)$, and $s=(-1,0)$, we obtain
$\alpha \geq \beta$.  From this argument and symmetry, we can conclude that
$f$ is jump M-convex if and only if $\alpha = \beta$.

Now suppose that $0 < \alpha < \beta$.  Then $f$ is not jump M-convex.
However, $\argmin f[-c]$ is a constant-parity jump system for each $c \in \RR\sp{2}$.
Indeed, 
$\argmin f[-c] = S \cap (2 \ZZ)\sp{2}$ for $c=(0,0)$, and
for $c \not=(0,0)$, $\argmin f[-c]$ 
is equal to a singleton or a set of three points 
like $\{ (0,0), (2,0), (4,0) \}$
lying on a horizontal or vertical line.
%% \finbox
\end{example}

However, such characterization is valid for jump M-convex functions
if $\dom f \subseteq  \{ 0,1 \}\sp{n}$.

\begin{theorem}  \label{THjm01fnargmin}
Assume that $\dom f$ is a constant-parity jump system contained in 
$\{ 0,1 \}\sp{n}$.
Then $f$ is jump M-convex
if and only if
$\argmin f[-c]$ is a constant-parity jump system for each $c \in \RR\sp{n}$.
\end{theorem}
\begin{proof}
A constant-parity jump system contained in 
$\{ 0,1 \}\sp{n}$
can be identified with an even delta-matroid,
and a function $f$ 
with $\dom f \subseteq  \{ 0,1 \}\sp{n}$
is jump M-convex
if and only if $-f$ is a valuated delta-matroid \cite{DW91valdel,Wen93pf}.
With this correspondence,
Theorem 2.2 of \cite{Mmax97} 
for valuated delta-matroids
is translated into this theorem.
\end{proof}

%% Correction 2020-02-08

Such characterization fails
for (global and local) discrete midpoint convexity,
as pointed out by \cite{TT20ddmc} only recently (after the submission of this paper).
That is, there is a function $f$ which is not
discrete midpoint convex but for which 
$\argmin f[-c]$ is discrete midpoint convex for every $c$.

%% 2019-12-03 / 2019-12-25 / 2020-01-08 / 2020-05-08

\section{Operations on Discrete Convex Sets}
\label{SCsetope}

In this section we consider operations on discrete convex sets.
The behavior of discrete convex sets 
with respect to the operations 
discussed below
is summarized in Table~\ref{TBoperation3dcsetZ} in Introduction.

\subsection{Direct sum}
\label{SCdirsumdset}

For two  sets 
$S_{1} \subseteq \ZZ\sp{n_{1}}$ and $S_{2} \subseteq \ZZ\sp{n_{2}}$, 
their {\em direct sum} is 
defined as
\begin{equation} \label{dirsumsetdef}
 S_{1} \oplus S_{2} = \{ (x,y) \mid x \in S_{1}, y \in S_{2}  \},
\end{equation}
which is a subset of  $\ZZ\sp{n_{1}+n_{2}}$.

In most cases it is obvious that the direct sum operation preserves
the discrete convexity in question.
However, this is not the case with 
multimodularity and discrete midpoint convexity.
We have the following proposition for the obvious cases.

\begin{proposition} \label{PRsetdirsumG}
The direct sum of two integrally convex sets
is an integrally convex set.
Similarly for 
\Lnat-convex sets,
L-convex sets,
\Mnat-convex sets,
M-convex sets,
simultaneous exchange jump systems,
and constant-parity jump systems.
%% \finbox
\end{proposition}

A multimodular set is defined with reference to 
an ordering of the underlying set.
When we consider the direct sum of two multimodular sets
$S_{1} \subseteq \ZZ\sp{n_{1}}$ and $S_{2} \subseteq \ZZ\sp{n_{2}}$,
we assume that the components of
$(x,y)$ 
are ordered naturally
with $x_{1},x_{2}, \ldots, x_{n_{1}}$ followed by $y_{1},y_{2}, \ldots, y_{n_{2}}$.
In this sense, it is more appropriate to 
regard an element $(x,y)$ of $S_{1} \oplus S_{2}$
as a {\em concatenation} of 
$x \in S_{1}$ and $y \in S_{2}$.

Proposition \ref{PRsetdirsumMult} below states that 
the direct sum $S_{1} \oplus S_{2}$ is also multimodular.
It is noted that 
this is a nontrivial statement,
since the definition of the multimodularity
of $S_{1} \oplus S_{2}$
involves the vector
$\unitvec{i}-\unitvec{i+1}$ for $i=n_{1}$
in \eqref{multimodirection1},
which does not appear in the definitions of 
the multimodularity of $S_{1}$ and $S_{2}$.

\begin{proposition} \label{PRsetdirsumMult}
The direct sum of two multimodular sets is multimodular.
\end{proposition}
\begin{proof}
The proof is given in Section \ref{SCmmdirsumprf}.
\end{proof}

In contrast, the direct sum of discrete midpoint convex sets
is not necessarily discrete midpoint convex,
as shown in Example~\ref{EXdmcsetdirsum} below.

\begin{example} \rm \label{EXdmcsetdirsum}
Let
\[
S_{1}  =  \{ (1, 0), (0, 1) \},
\qquad
S_{2}  =  \ZZ,
\]
for which
$S_{1} \oplus S_{2}= \{ (1, 0, t), (0, 1, t) \mid t \in \ZZ \}$.
The sets $S_{1}$ and $S_{2}$ are both discrete midpoint convex,
whereas  $S_{1} \oplus S_{2}$ is not.
Indeed, for $x=(1, 0, 2)$ and $y=(0, 1, 0)$ in $S_{1} \oplus S_{2}$,
we have $\| x - y \|_{\infty} = 2$,
$(x+y)/2 = (1/2, 1/2, 1)$,
for which
$\left\lceil (x+y)/2 \right\rceil = (1, 1, 1) \not\in S_{1} \oplus S_{2}$,
and
$\left\lfloor (x+y)/2 \right\rfloor = (0, 0, 1) \not\in S_{1} \oplus S_{2}$.
%% \finbox
\end{example}

\subsection{Splitting}
\label{SCsplitset}

Suppose that we are given a family 
$\{ U_{1},  U_{2}, \dots , U_{n} \}$
of  disjoint nonempty sets
indexed by $N = \{ 1, 2, \dots , n\}$. 
Let $m_{i}= |U_{i}|$ for $i=1,2,\ldots, n$ and define 
$m= \sum_{i=1}\sp{n} m_{i}$,
where $m \geq n$.
For a set $S \subseteq \ZZ\sp{n}$, 
the subset of $\ZZ\sp{m}$
defined by 
\begin{equation} \label{splitsetdef}
 T = \{ (y_{1}, y_{2}, \dots , y_{n}) \in  \ZZ\sp{m} \mid
  y_{i} \in \ZZ\sp{m_{i}}, \  x_{i} = y_{i}(U_{i}) \ \ (i \in N) , \ x  \in S   \}
\end{equation}
is called the {\em splitting} of $S$ by
$\{ U_{1},  U_{2}, \dots , U_{n} \}$.
A splitting is called an {\em elementary splitting}
if $|U_{k}|=2$ for some $k$ and $|U_{i}| =1$ for other $i\not= k$.
For example,
$T = \{ (y_{1}, y_{2}, y_{3}) \in  \ZZ\sp{3} \mid (y_{1}, y_{2}+ y_{3}) \in S  \}$
is an elementary splitting of $S \subseteq \ZZ\sp{2}$.
Any (general) splitting can be obtained by repeated applications of elementary splittings. 
It should be clear that the definition of splitting by \eqref{splitsetdef}
is consistent with the definition, given in Introduction,
in terms of the graph in Fig.~\ref{FGbiparmatroid} (b).

M-convexity  and its relatives are well-behaved 
with respect to the splitting operation, which is easy to see.

\begin{proposition} \label{PRsplitsetMMnat}
\quad

%% To Typesetter:  Please do NOT change the labels "(1)", "(2)", ....  below, 
%% as they are cited at other places.

\noindent
{\rm (1)} 
The splitting of an \Mnat-convex set is \Mnat-convex.

\noindent
{\rm (2)} 
The splitting of an M-convex set is M-convex.

\noindent
{\rm (3)} 
The splitting of a simultaneous exchange jump system is a simultaneous exchange jump system.

\noindent
{\rm (4)} 
The splitting of a constant-parity jump system is a constant-parity jump system.
\end{proposition}

The splitting operation has never been investigated 
for integrally convex sets and multimodular sets. 
For integrally convex sets we can show the following.

\begin{proposition} \label{PRsplitsetIC}
The splitting of an integrally convex set is integrally convex.
\end{proposition}
\begin{proof}
The proof is given in Section \ref{SCicvsetsplitprf}.
\end{proof}

In the definition of multimodularity, the ordering of 
the components of a vector is crucial.
Accordingly, 
in defining the splitting operation for multimodular sets,
we assume that 
the components of vector $y \in \ZZ\sp{m}$
are ordered naturally, 
first the $m_{1}$ components of $y_{1}$, 
then the $m_{2}$ components of $y_{2}$, etc., 
and finally the $m_{n}$ components of $y_{n}$.

\begin{proposition} \label{PRsplitsetMult}
The splitting of a multimodular set is multimodular
(under the natural ordering of the elements).
\end{proposition}
\begin{proof}
The proof is given in Section \ref{SCmmsplitprf}.
%%\qed
\end{proof}

Other kinds of discrete convexity
are not compatible with the splitting operation.
The splitting of an integer box is not necessarily an integer box.
Similarly,
the splitting of an 
\Lnat-convex
(resp., L-convex, discrete midpoint convex)
set
is not necessarily
\Lnat-convex
(resp., L-convex, discrete midpoint convex).
See Examples \ref{EXpointsplit} and \ref{EXlsetsplit}.

\begin{example} \rm \label{EXpointsplit}
The elementary splitting of a singleton set $S = \{ 0 \}$
is given by $T = \{ (t,-t) \mid t \in \ZZ \}$.
The set $S$ is an integer box but $T$ is not.
Also $S$ is an \Lnat-convex set but $T$ is not.
%% \finbox
\end{example}

\begin{example} \rm \label{EXlsetsplit}
The set
$S = \{ x \in \ZZ\sp{2}  \mid x_{1} = x_{2} \}$
is an L-convex set.
The elementary splitting of $S$ at the second component is given by
$T = \{ y \in \ZZ\sp{3}  \mid  y_{1} = y_{2} + y_{3} \}$.
This set is not L-convex since
the vector $y+\vecone$ does not belong to $T$ for $y \in T$.
%% \finbox
\end{example}

\subsection{Aggregation}
\label{SCaggrset}

Let 
$\mathcal{P}= \{ N_{1},  N_{2}, \dots , N_{m} \}$
 be a  partition of $N = \{ 1,2, \ldots, n \}$ into disjoint nonempty subsets,
i.e.,
$N = N_{1} \cup N_{2} \cup \dots \cup N_{m}$
and
$N_{i} \cap N_{j} = \emptyset$ for $i \not= j$.
We have $m \leq n$.
For a set $S \subseteq \ZZ\sp{n}$ 
the subset of $\ZZ\sp{m}$ defined by 
\begin{equation} \label{aggrsetdef}
 T = \{ (y_{1}, y_{2}, \dots , y_{m}) \in  \ZZ\sp{m} \mid
   y_{j} = x(N_{j})  \ (j=1,2,\ldots,m), \ x \in S  \}
\end{equation}
is called the {\em aggregation} of $S$ by $\mathcal{P}$.
An aggregation with $m=n-1$
is called an {\em elementary aggregation},
in which 
$|N_{k}|=2$ for some $k$ and $|N_{j}| =1$ for other $j\not= k$.
For example,
$T = \{ (y_{1}, y_{2}) \in  \ZZ\sp{2} \mid 
y_{1} = x_{1},
y_{2} = x_{2} + x_{3}
\mbox{ for some } (x_{1}, x_{2}, x_{3}) \in S  \}$
is an elementary aggregation of $S \subseteq \ZZ\sp{3}$.
Any (general) aggregation can be obtained by repeated applications of elementary aggregations. 
It should be clear that the definition of aggregation by \eqref{aggrsetdef}
is consistent with the definition, given in Introduction,
in terms of the graph in Fig.~\ref{FGbiparmatroid} (c).

It is known that M-convexity  and its relatives are well-behaved 
with respect to the aggregation operation.

\begin{proposition} \label{PRaggrset}
\quad

%% To Typesetter:  Please do NOT change the labels "(1)", "(2)", ....  below, 
%% as they are cited at other places.

\noindent
{\rm (1)} 
The aggregation of an integer box is an integer box.

\noindent
{\rm (2)} 
The aggregation of an \Mnat-convex set is \Mnat-convex.

\noindent
{\rm (3)} 
The aggregation of an M-convex set is M-convex.

\noindent
{\rm (4)} 
The aggregation of a simultaneous exchange jump system is a simultaneous exchange jump system.

\noindent
{\rm (5)} 
The aggregation of a constant-parity jump system is a constant-parity jump system.
%% \finbox
\end{proposition}
\begin{proof}
(1) The aggregation of an integer box 
$\{ x \in \ZZ\sp{n} \mid a_{i} \leq x_{i} \leq b_{i} \  (i=1,2,\ldots,n) \}$ 
is given by 
$\{ y \in \ZZ\sp{m} \mid a(N_{j}) \leq y_{j} \leq b(N_{j}) \  (j=1,2,\ldots,m) \}$,
which is an integer box.
The aggregation operations for M-convex and \Mnat-convex sets
in Parts (2) and (3) 
are well known in polymatroid/submodular function theory 
(see, e.g., \cite[Section 3.1(d)]{Fuj05book}).
The aggregation operation for (general) jump systems was considered by 
Kabadi and Sridhar \cite{KS05jump}.
Part~(5) for constant-parity jump systems
follows from this,
since $\sum_{j =1}\sp{m} y_{j} =  \sum_{i =1}\sp{n} x_{i}$ 
if $y_{j}=x(N_{j})$ $(j=1,2,\ldots,m)$. 
Part~(4) for simultaneous exchange jump systems 
can be derived from Part~(5) for constant-parity jump systems
on the basis of their relation 
\eqref{ftildeJ} in Section~\ref{SCjumpMdef}
by specializing the proof of \cite[Lemma 4.5]{Mmnatjump19}
to indicator functions.
\end{proof}

We point out here that 
other kinds of discrete convexity
are not compatible with the aggregation operation
by presenting counterexamples, as follows.

\begin{itemize}
\item
The aggregation of an integrally convex set
is not necessarily integrally convex
(Example~\ref{EXicvsetaggr}).

\item
The aggregation of an \Lnat-convex set
is not necessarily \Lnat-convex
(Example~\ref{EXlnatsetaggr}).

\item
The aggregation of an L-convex set
is not necessarily L-convex 
(Example~\ref{EXlsetaggr}).

\item
The aggregation of a multimodular set
is not necessarily multimodular
(Example~\ref{EXmmsetaggr}).

\item
The aggregation of a 
discrete midpoint convex set 
is not necessarily 
discrete midpoint convex
(Example~\ref{EXicvsetaggr}).
\end{itemize}

\begin{example} \rm \label{EXicvsetaggr}
The set
\[
S = \{ (0,0,1,0), (0,0,0,1), (1,1,1,0), (1,1,0,1) \}
\]
is an integrally convex set.
For the partition of $N = \{ 1,2,3,4 \}$
into $N_{1} = \{ 1,3 \}$ and $N_{2} = \{ 2,4 \}$,
the aggregation of $S$
by $\{ N_{1}, N_{2} \}$ 
is given by
\[
T = \{ (1,0), (0,1), (2,1), (1,2) \},
\]
which is not integrally convex.
The set $S$ is also 
discrete midpoint convex, but $T$ is not.
\OMIT{%%%%%%%%%%%%%%%%%%%%%%%%%%%%%%%%%%%%%
In view of Remark~\ref{RMsetaggruseMinkow},
it may be worth mentioning 
that this example is closely related to
the fact that the Minkowski sum of two integrally convex sets
$S_{1} = \{ (0,0), (1,1) \}$
and
$S_{2} = \{ (1,0), (0,1) \}$
is not integrally convex.
The set $S$ above is the direct sum of these sets
and $T$ is their Minkowski sum,
that is, $S = S_{1} \oplus S_{2}$ and $T = S_{1} +S_{2}$.
}%%% OMIT %%%%%%%%%%%%%%%%%%%%%%%%
%% \finbox
\end{example}

\begin{example} \rm \label{EXlnatsetaggr}
The set
\begin{equation}  \label{lnatsetaggDim6}
S = \{ (0,0,0,0,0,0), (0,0,0,0,1,1), (1,1,0,0,0,0), (1,1,0,0,1,1) \}
\end{equation}
is an \Lnat-convex set.
For the partition of $N = \{ 1,2, \ldots,6 \}$
into three pairs $N_{1} = \{ 1,4 \}$, $N_{2} = \{ 2,5 \}$, and $N_{3} = \{ 3,6 \}$,
the aggregation of $S$
by $\{ N_{1}, N_{2}, N_{3} \}$
is given by
\[
T  = \{(0, 0, 0), (0, 1, 1), (1, 1, 0), (1, 2, 1)\} ,
\]
which is not \Lnat-convex.
Indeed, for $x=(0, 1, 1)$ and $y=(1, 1, 0)$ in $T$,
we have
$(x+y)/2 = (1/2, 1, 1/2)$,
for which
$\left\lceil (x+y)/2 \right\rceil = (1, 1, 1) \not\in T$,
and
$\left\lfloor (x+y)/2 \right\rfloor = (0, 1, 0) \not\in T$.
Therefore, $T$ is not \Lnat-convex.
\OMIT{%%%%%%%%%%%%%%%%%%%%%%%%%%%%%%%%%%%%%
In view of  Remark~\ref{RMsetaggruseMinkow},
it may be worth mentioning 
that this example is closely related to
the fact 
that the Minkowski sum of two \Lnat-convex sets
$S_{1} =  \{(0, 0, 0), (1, 1, 0)\}$
and
$S_{2} =$ $ \{(0, 0, 0)$, $(0, 1, 1)\}$
is not \Lnat-convex.
The set $S$ above is the direct sum of these sets
and $T$ is their Minkowski sum,
that is, $S = S_{1} \oplus S_{2}$ and $T = S_{1} +S_{2}$.
}%%% OMIT %%%%%%%%%%%%%%%%%%%%%%%%
%% \finbox
\end{example}

\begin{example} \rm \label{EXlsetaggr}
(This is an adaptation of Example~\ref{EXlnatsetaggr} to L-convex sets.)
Let
$S_{1} =  \{(0, 0, 0, 0)+\alpha \vecone, (1, 1, 0, 0)+\alpha \vecone \mid \alpha \in \ZZ \}$
and
$S_{2} =  \{(0, 0, 0, 0)+\alpha \vecone, (0, 1, 1, 0)+\alpha \vecone \mid \alpha \in \ZZ \}$
with $\vecone = (1,1,1,1)$,
and define 
$S = S_{1} \oplus S_{2} \subseteq \ZZ\sp{8}$. 
This set $S$ is L-convex.
For the partition of $N = \{ 1,2, \ldots,8 \}$
into four pairs $N_{j} = \{ j, j+4 \}$ ($j=1,2,3,4)$,
the aggregation of $S$
is given by
\[
T  = \{(0, 0, 0, 0)+\alpha \vecone, (0, 1, 1, 0)+\alpha \vecone, 
       (1, 1, 0, 0)+\alpha \vecone, (1, 2, 1, 0)+\alpha \vecone \mid \alpha \in \ZZ \},
\]
which is not L-convex,
since for the elements 
$x=(0, 1, 1, 0)$ and $y=(1, 1, 0, 0)$ of $T$,
we have
$\left\lceil (x+y)/2 \right\rceil = (1, 1, 1, 0) \not\in T$
and
$\left\lfloor (x+y)/2 \right\rfloor = (0, 1, 0, 0) \not\in T$.
%% \finbox
\end{example}

\begin{example} \rm \label{EXmmsetaggr}
Here is an example of the aggregation of multimodular sets.
For the \Lnat-convex set $S$ in \eqref{lnatsetaggDim6}
(Example \ref{EXlnatsetaggr}),
let 
$\tilde S = \{ D x \mid x \in S \}$
be the multimodular set
corresponding to $S$,
where $D$ is the matrix defined in \eqref{matDdef}
in Section \ref{SCmultimod}.
That is,
\[
\tilde S 
 = \{ (0,0,0,0,0,0), (0,0,0,0,1,0), (1,0,-1,0,0,0), (1,0,-1,0,1,0) \}.
\]
For the partition of $N = \{ 1,2, \ldots,6 \}$
into three pairs $N_{1} = \{ 1,4 \}$, $N_{2} = \{ 2,5 \}$, and $N_{3} = \{ 3,6 \}$,
the aggregation of $\tilde S$ 
is given by
\[
\tilde T  = \{(0, 0, 0), (0, 1, 0), (1, 0, -1), (1, 1, -1)\} .
\]
This set $\tilde T$ is not multimodular.
We can check this directly or by detecting 
that the transformed set 
\[
T = \{ D\sp{-1} x \mid x \in \tilde T \}
 = \{(0, 0, 0), (0, 1, 1),  (1, 1, 0), (1, 2, 1)\}
\]
is not \Lnat-convex.
Indeed, $x=(0, 1, 1)$ and $y=(1, 1, 0)$ in $T$,
we have
$\left\lceil (x+y)/2 \right\rceil = (1, 1, 1) \not\in T$
and
$\left\lfloor (x+y)/2 \right\rfloor = (0, 1, 0) \not\in T$.
%% \finbox
\end{example}

\begin{remark} \rm \label{RMsetaggruseMinkow}
The
{\em Minkowski sum}
of two sets $S_{1}$, $S_{2} \subseteq \ZZ\sp{n}$ 
means the subset of $\ZZ\sp{n}$ 
defined by
\begin{equation} \label{minkowsumZdef}
S_{1}+S_{2} = 
\{ x + y \mid x \in S_{1}, \  y \in S_{2} \} ,
\end{equation}
which is useful and important in applications.
The Minkowski sum can be realized through a combination
of direct sum and aggregation operations.
%%Given two subsets  $S_{1}$ and $S_{2}$ of $\ZZ\sp{n}$,
We first form their direct sum
$S = S_{1} \oplus S_{2} \subseteq \ZZ\sp{2n}$.
The underlying set of $S$ is the union of two disjoint copies of 
$\{ 1,2,\ldots, n \}$, 
which we denote by
$\{ \psi_{1}(i) \mid i = 1,2,\ldots, n \} 
 \cup \{ \psi_{2}(i) \mid i = 1,2,\ldots, n \}$. 
Consider the partition of this underlying set 
into the pairs $\{ \psi_{1}(i), \psi_{2}(i) \}$ 
of corresponding elements.
Then the aggregation of $S$ coincides with 
the Minkowski sum $S_{1} + S_{2}$.
%% \finbox
\end{remark}

\subsection{Transformation by networks}
\label{SCnettransset}

In this section, we consider the transformation of a discrete (convex) set 
through a network. 
Let $G=(V, A; U, W)$ be a directed graph with vertex set $V$, arc set $A$, 
entrance set $U$, and exit set $W$, where $U$ and $W$ are disjoint subsets of $V$
(cf., Fig.~\ref{FGnettransSet}).
For each arc $a\in A$, 
an integer interval $[\ell(a), u(a)]_{\ZZ}$ is given 
as the capacity constraint,
where 
$\ell(a) \in \ZZ \cup \{ -\infty \}$
and
$u(a) \in \ZZ \cup \{ +\infty \}$.

We consider an integral flow $\xi: A \to \ZZ$
that satisfies 
the capacity constraint
on arcs:
\begin{equation} \label{nettransSetcapa}
\ell(a) \leq \xi(a) \leq u(a)
\qquad
(a \in A)
\end{equation}
and the flow-conservation at internal vertices:
\begin{equation} \label{nettransSetconserve}
 \sum_{a : \ a \text{ leaves } v} \xi (a) 
- \sum_{a : \ a \text{ enters } v} \xi (a) = 0
\qquad (v \in V \setminus ( U \cup W) ) .
\end{equation}
For $v \in V$ we use notation
\begin{equation} \label{nettransSetboundary}
\partial \xi (v) = \sum_{a : \ a \text{ leaves } v} \xi (a) 
- \sum_{a : \ a \text{ enters } v} \xi (a) ,
\end{equation}
which means the net flow-supply from outside of the network at vertex $v$.
Accordingly, $\partial \xi \in \ZZ\sp{V}$ is the vector of net supplies.  
The restriction of $\partial \xi$ to $U$ is denoted by
$\partial \xi | U$, that is,
$x = \partial \xi | U$ is a vector with components indexed by $U$
such that
$x(v) = \partial \xi (v)$
for $v \in U$. 
Similarly we define $\partial \xi | W \in \ZZ\sp{W}$.

%%%  FIGURE %%%%%%%%%%%%%%%%%%
%%\input{fg2NettransSet}
\begin{figure}\begin{center}
 \includegraphics[width=0.6\textwidth,clip]{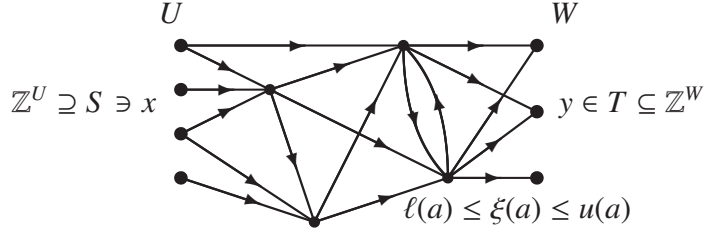}
\caption{Transformation of a discrete convex set by a network}
\label{FGnettransSet}
%%\hfill \memo{fg2NettransSet.tex}
\end{center}\end{figure}
%%%  FIGURE %%%%%%%%%%%%%%%%%%

Given a set $S \subseteq \ZZ\sp{U}$
of integer vectors on the entrance set $U$, 
we consider the set $T \subseteq \ZZ\sp{W}$ 
of integer vectors $y$ on the exit set $W$
for which 
there is a feasible flow $\xi$
such that 
the net supply vector on $U$ belongs to the given set $S$
(i.e., $\partial \xi | U \in S$)
and the net supply vector on $W$ coincides with $-y$ 
(i.e., $\partial \xi | W = -y$).
That is,
\begin{align}
T = \{ y \in \ZZ\sp{W} & \mid  
\mbox{\rm there exists $\xi \in \ZZ\sp{A}$ satisfying 
\eqref{nettransSetcapa},  \eqref{nettransSetconserve}, } 
 \notag \\ & \quad
 \mbox{$\partial \xi | U \in S$, and $\partial \xi | W = -y$ }
\}.
\label{nettransSetdef}
\end{align}
We regard $T$ as a result of {\em transformation} (or {\em induction}) of $S$ by the network. 
It is assumed that $T$ is nonempty.

\OMIT{%%%%%%%%%%%%%%%%%%%%%%%%%%%%%%%%%
\begin{remark} \rm  \label{RMsplaggrminkowNetset}
\memo{Moved to Introduction}
Splitting, aggregation, and the Minkowski sum
can be regarded as special cases of the transformation
by means of bipartite networks,
as shown in Fig.~\ref{FGsplaggrminkow}.
For the Minkowski sum we first make the direct sum, to which aggregation 
is applied (cf., Remark \ref{RMsetaggruseMinkow}).
%% \finbox
\end{remark}
} %%%%%%%% OMIT %%%%%%%%%

It is known that 
M-convexity  and its relatives are well-behaved 
with respect to the network induction.

\begin{theorem} \label{THnetindset}
\quad

%% To Typesetter:  Please do NOT change the labels "(1)", "(2)", ....  below, 
%% as they are cited at other places.

\noindent
{\rm (1)} 
The network induction of an \Mnat-convex set is \Mnat-convex.

\noindent
{\rm (2)} 
The network induction of an M-convex set is M-convex.

\noindent
{\rm (3)} 
The network induction of a simultaneous exchange jump system is a simultaneous exchange jump system.

\noindent
{\rm (4)} 
The network induction of a constant-parity jump system is a constant-parity jump system.
%% \finbox
\end{theorem}

\begin{remark} \rm \label{RMnetindset}
Here is a supplement to Theorem~\ref{THnetindset}.
These statements are reformulations of known facts 
in matroid/polymatroid/submodular function theory 
(see, e.g., \cite{BouC95,Fuj05book,KS05jump,Sch03}).
%% added 2020-05-04 
Parts~(1) and (2) for \Mnat-convex and M-convex sets
are variants of the statement that an integral polymatroid
(defined in terms of independent sets or bases) 
is transformed to another integral polymatroid
through Menger-type linkings in a given directed graph.
%%%%%%%% end of addition
Part~(3) for simultaneous exchange jump systems 
is a special case of \cite[Theorem 4.12]{Mmnatjump19}.
Part~(4) for constant-parity jump systems
is a special case of \cite[Theorem 14]{KMT07jump}.
%% \finbox
\end{remark}

In contrast, other kinds of discrete convexity
are not compatible with the network induction.
The network induction of an integer box is not necessarily an integer box.
Similarly,
the network induction of an 
integrally convex
(resp., \Lnat-convex, L-convex, multimodular, discrete midpoint convex)
set
is not necessarily 
integrally convex
(resp., \Lnat-convex, L-convex, multimodular, discrete midpoint convex).
Note that these statements are immediate from 
the corresponding statements for splitting and aggregation
in Sections \ref{SCsplitset} and \ref{SCaggrset},
since the network induction is more general than those operations.

%% 2019-12-03 / 2019-12-25 / 2020-05-08

\section{Operations on Discrete Convex Functions}
\label{SCfnope}

In this section we consider operations on discrete convex functions.
The behavior of discrete convex functions 
with respect to the operations discussed below
is summarized in Table~\ref{TBoperation3dcfnZ} in Introduction.

\subsection{Direct sum}
\label{SCdirsumconvfn}

The {\em direct sum} of two functions 
$f_{1}: \ZZ\sp{n_{1}} \to \RR \cup \{ +\infty \}$ and
$f_{2}: \ZZ\sp{n_{2}} \to \RR \cup \{ +\infty \}$
is a function 
$f_{1} \oplus f_{2}: \ZZ\sp{n_{1}+n_{2}} \to \RR \cup \{ +\infty \}$
defined as
\begin{equation} \label{fndirsumdef}
(f_{1} \oplus f_{2})(x,y)= f_{1}(x) + f_{2}(y)
\qquad (x \in \ZZ\sp{n_{1}}, y \in \ZZ\sp{n_{2}}) .
\end{equation}
The effective domain of the direct sum is equal to the
direct sum of
 the effective domains of the given functions, that is,
\begin{equation} \label{fndirsumdom}
 \dom (f_{1} \oplus  f_{2}) = \dom f_{1} \oplus \dom f_{2}.
\end{equation}
For two sets 
$S_{1} \subseteq \ZZ\sp{n_{1}}$ and $S_{2} \subseteq \ZZ\sp{n_{2}}$, 
the direct sum of their indicator functions
$\delta_{S_{1}}$ and $\delta_{S_{2}}$
coincides with 
the indicator function of their direct sum
$S_{1} \oplus S_{2}$,
that is,
\[
\delta_{S_{1}} \oplus  \delta_{S_{2}} = \delta_{S_{1} \oplus S_{2}}.
\]

In most cases it is obvious that the direct sum operation preserves
the discrete convexity in question.
However, this is not the case with 
multimodularity and discrete midpoint convexity.
We have the following proposition for the obvious cases.

\begin{proposition} \label{PRfndirsumG}
The direct sum operation \eqref{fndirsumdef} for functions preserves
separable convexity,
integral convexity,
\Lnat-convexity, L-convexity,
\Mnat-convexity, M-convexity,
and jump \Mnat-convexity, and jump M-convexity.
%% \finbox
\end{proposition}

Proposition \ref{PRfndirsumMult} below states that 
the direct sum $f_{1} \oplus f_{2}$ 
of two multimodular functions $f_{1}$ and $f_{2}$ 
is also multimodular.
It is noted that 
this is a nontrivial statement,
since the definition of the multimodularity
of $f_{1} \oplus f_{2}$
involves the vector
$\unitvec{i}-\unitvec{i+1}$ for $i=n_{1}$
in \eqref{multimodirection1},
which does not appear in the definitions of 
the multimodularity of $f_{1}$ and $f_{2}$.
Just as for the direct sum of multimodular sets,
it is assumed that the components of $(x,y)$   
in the definition \eqref{fndirsumdef} of 
$f_{1} \oplus f_{2}: \ZZ\sp{n_{1} + n_{2}} \to \RR \cup \{ +\infty \}$
are ordered naturally
with $x_{1},x_{2}, \ldots, x_{n_{1}}$ followed by $y_{1},y_{2}, \ldots, y_{n_{2}}$.

\begin{proposition} \label{PRfndirsumMult}
The direct sum of two multimodular functions is multimodular.
\end{proposition}
\begin{proof}
The proof is given in  Section \ref{SCmmdirsumprf}.
\end{proof}

In contrast, the direct sum of globally (resp., locally) 
 discrete midpoint convex functions
is not necessarily 
globally (resp., locally) discrete midpoint convex.
This is shown already by Example~\ref{EXdmcsetdirsum},
and the following example 
gives $f_{1}$ and $f_{2}$ that are finite-valued at every integer point.

\begin{example} \rm \label{EXdirsumdmcfn}
Let $f_{1}:\ZZ\sp{2} \to \RR$
and 
$f_{2}:\ZZ \to \RR$ be defined by 
\[
 f_{1}(x_{1},x_{2}) = {x_{1}}\sp{2}  + x_{1} x_{2} + {x_{2}}\sp{2} ,
\qquad
 f_{2}(x_{3}) = 0.
\]
While $f_{1}$ and $f_{2}$ are (globally and locally) discrete midpoint convex,
their direct sum 
\[
g(x_{1},x_{2},x_{3}) = (f_{1} \oplus f_{2})(x_{1},x_{2},x_{3})  
 = {x_{1}}\sp{2}  + x_{1} x_{2} + {x_{2}}\sp{2} 
\]
is not (globally and locally) discrete midpoint convex.
Indeed, for $x = (1,0,0)$, $y = (0,1,2)$, we have
$\| x - y \|_{\infty} = 2$, \
$u=\left\lceil \frac{x+y}{2} \right\rceil = (1,1,1)$,
$v = \left\lfloor \frac{x+y}{2} \right\rfloor = (0,0,1)$, and
$g(x) + g(y) = 1 + 1  < g(u) + g(v) = 3 + 0$.
%% \finbox
\end{example}

\subsection{Splitting}
\label{SCsplitfn}

Suppose that we are given a family 
$\{ U_{1},  U_{2}, \dots , U_{n} \}$
of  disjoint nonempty sets
indexed by $N = \{ 1, 2, \dots , n\}$. 
Let $m_{i}= |U_{i}|$ for $i=1,2,\ldots, n$ and define 
$m= \sum_{i=1}\sp{n} m_{i}$, where $m \geq n$.
For a function $f: \ZZ \sp{n} \to \RR \cup \{+\infty\}$, 
the {\em splitting} of $f$ 
by $\{ U_{1},  U_{2}, \dots , U_{n} \}$
is defined as a function
$g: \ZZ \sp{m} \to \RR \cup \{+\infty\}$ given by
\begin{equation} \label{splitfndef}
 g(y_{1}, y_{2}, \dots , y_{n}) 
  = f( y_{1}(U_{1}), y_{2}(U_{2}), \dots , y_{n}(U_{n})), 
\end{equation}
where, for each $i \in N$, $y_{i} = ( y_{ij} \mid  j \in U_{i} )$ is 
an integer vector of dimension $m_{i}$
and $y_{i}(U_{i}) = \sum\{  y_{ij} \mid j \in U_{i} \}$ is 
the component sum of vector  $y_{i} \in \ZZ\sp{m_{i}}$. 
If $m = n+1$
(in which case we have $|U_{k}|=2$ for some $k$ and $|U_{i}| =1$ for other $i\not= k$),
this is called an {\em elementary splitting}.
For example,
$g(y_{1}, y_{2}, y_{3}) = f(y_{1}, y_{2}+ y_{3})$
is an elementary splitting of $f$.
Any (general) splitting can be obtained by repeated applications of elementary splittings. 
It should be clear that the definition of splitting by \eqref{splitfndef}
is consistent with the definition, given in Introduction,
in terms of the graph in Fig.~\ref{FGbiparmatroid} (b).

It is known that 
M-convexity  and its relatives
are well-behaved with respect to the splitting operation.

\begin{proposition} \label{PRsplitfnMMnat}
\quad

%% To Typesetter:  Please do NOT change the labels "(1)", "(2)", ....  below, 
%% as they are cited at other places.

\noindent
{\rm (1)} 
The splitting of an \Mnat-convex function is \Mnat-convex.

\noindent
{\rm (2)} 
The splitting of an M-convex function is M-convex.

\noindent
{\rm (3)} 
The splitting of a jump \Mnat-convex function is jump \Mnat-convex.

\noindent
{\rm (4)} 
The splitting of a jump M-convex function is jump M-convex.
\end{proposition}
\begin{remark} \rm \label{RMfnsplitbib}
Here is a supplement to Proposition~\ref{PRsplitfnMMnat}.
The splitting operation for discrete convex functions
is considered explicitly in 
\cite{KMT07jump}
for jump M-convex functions.
This is given in Part (4).
As the splitting operation is a special case of 
the transformation by a bipartite network (cf., 
Remark \ref{RMsplaggrconvolNetfn}),
Parts (1) and (2) 
for \Mnat-convex and M-convex functions
follow from the previous results on the 
network induction 
for \Mnat-convex and M-convex functions
stated in \cite[Theorem 9.26]{Mdcasiam}.
Part (3) for jump \Mnat-convex functions
is derived in \cite{Mmnatjump19}
from (4) for jump M-convex functions.
%% \finbox
\end{remark}

The splitting operation has never been investigated 
for integrally convex functions and multimodular functions. 
For integrally convex functions we can show the following.

\begin{proposition} \label{PRsplitfnIC}
The splitting of an integrally convex function is integrally convex.
\end{proposition}
\begin{proof}
The proof is given in Section \ref{SCicvfnsplitprf}.
%%\qed
\end{proof}

In the definition of multimodularity, the ordering of 
the components of a vector is crucial.
Accordingly, 
in defining the splitting operation for multimodular functions,
we assume that the components of vector $y \in \ZZ\sp{m}$
are ordered naturally, 
first the $m_{1}$ components of $y_{1}$, 
then the $m_{2}$ components of $y_{2}$, etc., 
and finally the $m_{n}$ components of $y_{n}$.

\begin{proposition} \label{PRsplitfnMult}
The splitting of a multimodular function is multimodular
(under the natural ordering of the elements).
\end{proposition}
\begin{proof}
The proof is given in Section \ref{SCmmsplitprf}.
%%\qed
\end{proof}

In contrast, L-convexity and its relatives 
are not compatible with the splitting operation.
That is, the splitting operation does not preserve
separable convexity,
\Lnat-convexity, L-convexity,
and (global, local)
discrete midpoint convexity.
This is immediate from 
the corresponding statements for the splitting of sets
in Section \ref{SCsplitset}.

\subsection{Aggregation}
\label{SCaggrfn}

Let 
$\mathcal{P} =  \{ N_{1},  N_{2}, \dots , N_{m} \}$
 be a  partition of $N = \{ 1,2, \ldots, n \}$ into disjoint (nonempty) subsets,
i.e.,
$N = N_{1} \cup N_{2} \cup \dots \cup N_{m}$
and
$N_{i} \cap N_{j} = \emptyset$ for $i \not= j$.
We have $m \leq n$.
For a function $f: \ZZ \sp{n} \to \RR \cup \{+\infty \}$, 
the {\em aggregation} of $f$ with respect to $\mathcal{P}$
is the function 
$g : \ZZ \sp{m} \to \RR \cup \{+\infty, -\infty\}$ 
defined by 
\begin{equation} \label{aggrfndef}
g(y_{1}, y_{2}, \dots , y_{m}) 
= \inf \{ f(x)   \mid x(N_{j}) = y_{j} \ (j=1,2,\ldots,m) \} ,
\end{equation}
where $y_{j} \in \ZZ$ for $j=1,2,\ldots,m$.
If $m = n-1$
(in which case we have $|N_{k}|=2$ for some $k$ and $|N_{j}| =1$ for other $j\not= k$),
 this is called an {\em elementary aggregation}.
For example,
$g(y_{1}, y_{2}) = \inf \{ f(x_{1}, x_{2}, x_{3}) 
   \mid x_{1} = y_{1}, \ x_{2} + x_{3} = y_{2} \}$
is an elementary aggregation of $f$.
Any (general) aggregation can be obtained by repeated applications of elementary aggregations. 
It should be clear that the definition of aggregation by \eqref{aggrfndef}
is consistent with the definition, given in Introduction,
in terms of the graph in Fig.~\ref{FGbiparmatroid} (c).

It is known that M-convexity  and its relatives are well-behaved 
with respect to the aggregation operation.

\begin{proposition} \label{PRaggrfnMMnat}
\quad

%% To Typesetter:  Please do NOT change the labels "(1)", "(2)", ....  below, 
%% as they are cited at other places.

\noindent
{\rm (1)} 
The aggregation of a separable convex function is separable convex.

\noindent
{\rm (2)} 
The aggregation of an \Mnat-convex function is \Mnat-convex.

\noindent
{\rm (3)} 
The aggregation of an M-convex function is M-convex.

\noindent
{\rm (4)} 
The aggregation of a jump \Mnat-convex function is jump \Mnat-convex.

\noindent
{\rm (5)} 
The aggregation of a jump M-convex function is jump M-convex.
%% \finbox
\end{proposition}

\begin{remark} \rm \label{RMfnaggrbib}
Here is a supplement to Proposition~\ref{PRaggrfnMMnat}.
The aggregation of a separable convex function
$\sum_{i=1}\sp{n} \varphi_{i}(x_{i})$ 
is given by a separable convex function
$\sum_{j=1}\sp{m} \psi_{j}(y_{j})$
with 
$\psi_{j}(y_{j}) = \inf \{ \sum_{i \in N_{j}} \varphi_{i}(x_{i})  \mid x(N_{j}) = y_{j} \}$,
where $y_{j} \in \ZZ$ for $j=1,2,\ldots,m$.
The aggregation operations for M-convex and \Mnat-convex functions
in (2) and (3) 
are given in 
\cite[Theorem 6.13 ]{Mdcasiam}
and
\cite[Theorem 6.15]{Mdcasiam},
respectively. 
Part~(5) for jump M-convex functions
is established in \cite{KMT07jump}
by a long proof.
Part~(4) for jump \Mnat-convex functions 
is derived 
in \cite{Mmnatjump19}
from (5) for jump M-convex functions.
%% \finbox
\end{remark}

In contrast, other kinds of discrete convexity
are not compatible with the aggregation operation.
That is, the aggregation operation does not preserve
integral convexity,
\Lnat-convexity,
L-convexity,
multimodularity,
and 
(global, local)
discrete midpoint convexity.
This is immediate from 
the corresponding statements for the aggregation of sets
in Section \ref{SCaggrset}.

\begin{remark} \rm \label{RMfbaggruseConvol}
The convolution of two functions can be
realized through a combination
of direct sum and aggregation operations. 
We recall that the (infimal) {\em convolution} of two functions
$f_{1}, f_{2}: \ZZ\sp{n} \to \RR \cup \{ +\infty \}$
is defined by
\begin{equation} \label{f1f2convdef}
(f_{1} \conv f_{2})(x) =
 \inf\{ f_{1}(y) + f_{2}(z) \mid x= y + z, \  y, z \in \ZZ\sp{n}  \}
\qquad (x \in \ZZ\sp{n}) ,
\end{equation}
where it is assumed that the infimum 
is bounded from below (i.e., $(f_{1} \conv f_{2})(x) > -\infty$
for every $x \in \ZZ\sp{n}$).
For the given functions
$f_{1}$ and $f_{2}$
we first form their direct sum
\[
f(x_{1}, x_{2}) = f_{1}(x_{1}) + f_{2}(x_{2}), 
\]
where $x_{1}, x_{2} \in \ZZ\sp{n}$. 
The underlying set of $f$ is the union of two disjoint copies of 
$\{ 1,2,\ldots, n \}$, 
which we denote by
$\{ \psi_{1}(i) \mid i = 1,2,\ldots, n \} 
 \cup \{ \psi_{2}(i) \mid i = 1,2,\ldots, n \}$. 
Consider the partition of this underlying set 
into the pairs $\{ \psi_{1}(i), \psi_{2}(i) \}$ 
of corresponding elements.
Then the aggregation of $f$ coincides with 
the convolution $f_{1} \Box f_{2}$.
%% \finbox
\end{remark}

\subsection{Transformation by networks}
\label{SCnettransfn}

In this section, we consider the transformation of a discrete (convex) function 
through a network. 
As in Section \ref{SCnettransset},
let $G=(V, A; U, W)$ be a directed graph with vertex set $V$, arc set $A$, 
entrance set $U$, and exit set $W$, where $U$ and $W$ are disjoint subsets of $V$
(cf., Fig.~\ref{FGnettransSet}).
For each arc $a\in A$, 
an integer interval $[\ell(a), u(a)]_{\ZZ}$ is given 
as the capacity constraint,
where 
$\ell(a) \in \ZZ \cup \{ -\infty \}$
and
$u(a) \in \ZZ \cup \{ +\infty \}$.
We consider an integral flow $\xi: A \to \ZZ$
that satisfies 
the capacity constraint 
\eqref{nettransSetcapa} on arcs
and the flow-conservation \eqref{nettransSetconserve} at internal vertices.
Recall notations
$\partial \xi \in \ZZ\sp{V}$, 
$\partial \xi | U \in \ZZ\sp{U}$, and $\partial \xi | W \in \ZZ\sp{W}$.

In addition, we assume that 
the cost of integer-flow $\xi$
is measured 
in each arc $a\in A$ 
in terms of a function 
$\varphi_a : \ZZ \to \RR \cup \{ +\infty \}$,
where $\dom \varphi_a = [\ell(a), u(a)]_{\ZZ}$
and $\varphi_a$ is (discrete) convex in the sense that
\begin{equation}  \label{arccostconv}
\varphi_{a}(t-1) + \varphi_{a}(t+1) \geq 2 \varphi_{a}(t)
\qquad (t \in \ZZ).
\end{equation}

Suppose we are given a function $f : \ZZ\sp{U} \to \RR \cup \{ +\infty \}$ associated 
with the entrance set $U$.
For each vector $y \in \ZZ\sp{W}$ on the exit set $W$, 
we define a function $g(y)$ as the minimum cost of a flow $\xi$
to meet the demand specification $\partial \xi | W = -y$ at the exit, 
where the cost of flow $\xi$ consists of two parts, 
the production cost $f(x)$
of $x = \partial \xi | U$ 
at the entrance and 
the transportation cost $\sum_{a\in A} \varphi_a (\xi (a))$  
at arcs; 
the sum of these is to be minimized over varying supply $x$ and flow $\xi$ subject to 
the supply-demand constraints
$\partial \xi | U = x$ and $\partial \xi | W = -y$
as well as 
the flow conservation constraint
\eqref{nettransSetconserve} at internal vertices.
That is, $g : \ZZ\sp{W} \to \RR \cup \{ + \infty, -\infty \}$
is defined as 
\begin{align}
g(y) = \inf_{x,\,  \xi} & \big\{ f(x) + \sum_{a\in A} \varphi_a (\xi (a)) 
\mid 
\mbox{\rm  $x \in \ZZ\sp{U}$ and $\xi \in \ZZ\sp{A}$  satisfy 
\eqref{nettransSetcapa},  \eqref{nettransSetconserve}, } 
 \notag \\ & \ \  
 \mbox{$\partial \xi | U = x$, and $\partial \xi | W = -y$ }
\}
\qquad \qquad (y \in \ZZ\sp{W}), 
\label{nettransfndef}
\end{align}
where $g(y) = + \infty$ if no such $(x ,\xi)$ exists.
It is assumed that 
the effective domain $\dom g$ is nonempty and that
the infimum is bounded from below 
(i.e., $g(y) > -\infty$ for every $y \in \ZZ\sp{W}$).
We regard $g$ as a result of {\em transformation} (or {\em induction}) of $f$ by the network.

\begin{remark} \rm  \label{RMsplaggrconvolNetfn}
Splitting, aggregation, and convolution 
can be regarded as special cases of the transformation
by means of bipartite networks
(cf., Fig.~\ref{FGbiparmatroid}).
For the convolution  
we use the bipartite graph (d) in Fig.~\ref{FGbiparmatroid}.
%% \finbox
\end{remark}

It is known that M-convexity  and its relatives are well-behaved 
with respect to the network induction.

\begin{theorem} \label{THnetindfn}
\quad

%% To Typesetter:  Please do NOT change the labels "(1)", "(2)", ....  below, 
%% as they are cited at other places.

\noindent
{\rm (1)} 
The network induction of an \Mnat-convex function is \Mnat-convex.

\noindent
{\rm (2)} 
The network induction of an M-convex function is M-convex.

\noindent
{\rm (3)} 
The network induction of a jump \Mnat-convex function is jump \Mnat-convex.

\noindent
{\rm (4)} 
The network induction of a jump M-convex function is jump M-convex.
\end{theorem}
\begin{proof} 
(4)
The proof for jump M-convex functions, given in \cite{KMT07jump},
is based on splitting and aggregation
(Propositions \ref{PRsplitfnMMnat} and \ref{PRaggrfnMMnat}),
and other simple operations
such as independent coordinate inversion,
restriction, and
addition of a separable convex function
treated in \cite[Propositions 4.3, 4.9, 4.14]{Msurvop19}.

(3) The proof for jump \Mnat-convex functions can be obtained as an adaptation
of the proof for jump M-convex functions,
as pointed out in \cite{Mmnatjump19}.
This is possible since 
splitting and aggregation
are allowed also for jump \Mnat-convex functions
by Propositions \ref{PRsplitfnMMnat} and \ref{PRaggrfnMMnat},
as well as
independent coordinate inversion,
restriction, and
addition of a separable convex function
(cf., \cite[Propositions 4.3, 4.9, 4.14]{Msurvop19}).

(2) Two kinds of proofs are known for M-convex functions. 
The first proof \cite{Mstein96}
uses a dual variable and a characterization of M-convexity of a function
in terms of its minimizers.
The second proof
\cite{Shi96indRep,Shi98ind} 
is an algorithmic proof,
which is described in \cite[Section 9.6.2]{Mdcasiam}. 
Yet another proof is possible, which
derives this as a corollary of Part (4) for jump M-convex functions.
Recall that an M-convex function is characterized as a jump M-convex function 
that has a constant-sum effective domain. 
If the given function $f$ is M-convex, then it is jump M-convex, and therefore,
$g$ is jump M-convex by Part (4). In addition, $\dom g$ is a constant-sum system,
since $\dom f$ is a constant-sum system and 
$\partial \xi(U) + \partial \xi(W) = 0$
by \eqref{nettransSetconserve}.
Therefore, $g$ is M-convex.

(1)
The proof for \Mnat-convex functions
can be obtained from Part (2) for M-convex functions
as follows.
Let $f$ be an \Mnat-convex function given on $U$.
Consider two new vertices $u_{0}$ and $w_{0}$
and an arc $(u_{0},w_{0})$,
and let
$\tilde U = U \cup \{ u_{0} \}$,
$\tilde W = W \cup \{ w_{0} \}$,
$\tilde V = V \cup \{ u_{0}, w_{0} \}$, 
$\tilde A = A \cup \{ (u_{0}, w_{0}) \}$, and
$\tilde G=(\tilde V, \tilde A; \tilde U, \tilde W)$.
For $a =(u_{0},w_{0})$
we define $\ell(a) = -\infty$, $u(a) = +\infty$,
and $\varphi_{a} \equiv 0$.
Let 
$\tilde{f}$ and $\tilde{g}$ be the functions
associated, respectively, with 
$f$ and $g$ as in \eqref{mfnmnatfnrelationvex},
where 
$\dom \tilde{f} \subseteq \{ x \in \ZZ\sp{\tilde U} \mid x(\tilde U)=0 \}$
and
$\dom \tilde{g} \subseteq \{ y \in \ZZ\sp{\tilde W} \mid y(\tilde W)=0 \}$.
If the function $g$ is induced from $f$ by $G$,
then  $\tilde g$ coincides with the function induced from $\tilde f$ by $\tilde G$.
Since $f$ is \Mnat-convex, $\tilde f$ is M-convex, and hence
$\tilde g$ is M-convex by Part (2).
This implies that $g$ is \Mnat-convex.
It is also possible to adapt the first and second proofs for M-convex functions
to \Mnat-convex functions.
\end{proof}

\begin{remark} \rm \label{RMnetindfn}
Here is a supplement to Theorem~\ref{THnetindfn}.
The network induction for discrete convex functions
is considered first by Murota \cite{Mstein96} 
for M-convex functions,
and stated also in 
\cite[Theorem 9.26]{Mdcasiam}. 
Part (1) for \Mnat-convex functions
is a variant thereof,
and stated in \cite[Theorem 9.26]{Mdcasiam}. 
Part~(4) for jump M-convex function is established in \cite{KMT07jump}
and Part (3) for jump \Mnat-convex functions
is derived therefrom in 
\cite{Mmnatjump19}. 
Theorem \ref{THnetindfn} here is a generalization 
of Theorem \ref{THnetindset} for discrete convex sets.
The transformation by networks can be generalized by replacing networks 
with poly-linking systems, 
and it is shown in \cite{KM07jumplink} that
the transformation by valuated integral poly-linking systems 
preserves M-convexity and jump M-convexity.
%% \finbox
\end{remark}

\begin{example}[{\cite[Note 9.31]{Mdcasiam}}]  \rm  \label{EXlaminarByNetInd}
A laminar convex function introduced in Section \ref{SCmnatfn}
can be constructed by means of the network induction.
As a concrete example, consider
\[
g(y_{1}, y_{2}, y_{3} ) = | y_{1} + y_{2} + y_{3} | + ( y_{1} + y_{2} )\sp{2} + {y_{3}}\sp{2} ,
\]
which is a laminar convex function of the form of \eqref{mnatfnlaminar}
with a laminar family 
$\mathcal{T} = \{ \  \{ 1,2,3 \},    \allowbreak     \{ 1,2 \}, \{ 1 \}, \{ 2 \}, \{ 3 \} \  \}$
and univariate convex functions
$\varphi_{123}(t) = |t|$,
$\varphi_{12}(t) = \varphi_{3}(t) = t\sp{2}$, and
$\varphi_{1}(t) = \varphi_{2}(t) = 0$.
For this function we consider the graph $G$, 
a rooted directed tree, depicted in Fig.~\ref{FGlaminartree3}.
Each vertex other than the root $u$ corresponds to a member of $\mathcal{T}$.
The entrance set $U$ is the singleton set of the root, 
i.e., $U = \{ u  \}$, and
the exit set $W$ is the set of the leaves, 
i.e., $W = \{ v_{1},v_{2},v_{3}  \}$.
The cost function $\varphi_{a}$ on each arc is determined
by the head of the arc; for example,
we have 
$\varphi_{123}(t) = |t|$ for arc $(u, v_{123})$
and
$\varphi_{12}(t) = t\sp{2}$ for arc $(v_{123}, v_{12})$.
Assume that 
the identically zero function $f \equiv 0$ is defined on the entrance set $U$.
Then the function induced from $f$ by $G$ coincides with the function
$g(y_{1},y_{2},y_{3} ) = g(-y_{1}, -y_{2}, -y_{3} )$.
Since $f \equiv 0$ is \Mnat-convex, Theorem \ref{THnetindfn} (1)
shows that $g(y_{1}, y_{2}, y_{3} )$ is \Mnat-convex.
\end{example}

%%%  FIGURE %%%%%%%%%%%%%%%%%%
%%\input{fg3laminartree}
\begin{figure}\begin{center}
 \includegraphics[width=0.4\textwidth,clip]{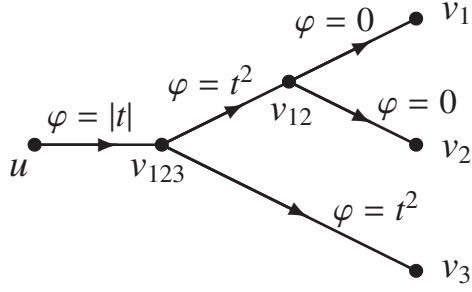}
 \caption{Network induction for a laminar convex function}
 \label{FGlaminartree3}
%%\hfill \memo{fg3laminartree}
\end{center}\end{figure}
%%%  FIGURE %%%%%%%%%%%%%%%%%%

In contrast, other kinds of discrete convexity
are not compatible with the network induction.
That is, the network induction does not preserve
separable convexity,
integral convexity,
\Lnat-convexity,
L-convexity,
multimodularity,
and 
(global, local)
discrete midpoint convexity.
Note that these statements are immediate from 
the corresponding statements for splitting and aggregation
in Sections \ref{SCsplitfn} and \ref{SCaggrfn},
since the network induction is more general than those operations.

%% 2019-12-03 / 2019-12-25 / 2020-05-08

\section{Proofs}
\label{SCproof}

In this section we give proofs for
Propositions \ref{PRsetdirsumMult},
\ref{PRsplitsetIC},
\ref{PRsplitsetMult},
\ref{PRfndirsumMult},
\ref{PRsplitfnIC}, and 
\ref{PRsplitfnMult}.
We deal with propositions concerning integral convexity
in the first two subsections,
and then those concerning multimodularity
in the following subsections, as follows:

\begin{itemize}
\item
Section \ref{SCicvsetsplitprf}:  Proposition \ref{PRsplitsetIC}
for splitting of integrally convex sets,

\item
Section \ref{SCicvfnsplitprf}: Proposition \ref{PRsplitfnIC}
for splitting of integrally convex functions,

\item
Section \ref{SCmmdirsumprf}:
Propositions \ref{PRsetdirsumMult} and \ref{PRfndirsumMult} 
for direct sum of multimodular sets and functions,

\item
Section \ref{SCmmsplitprf}: 
Propositions \ref{PRsplitsetMult} and \ref{PRsplitfnMult}
for splitting of multimodular sets and functions.
\end{itemize}

\subsection{Proof for the splitting of integrally convex sets}
\label{SCicvsetsplitprf}

Here is a proof of Proposition \ref{PRsplitsetIC}
concerning the splitting of an integrally convex set $S$.
It suffices to consider an elementary splitting. 
Specifically we consider the splitting of 
the first variable $x_{1}$ of $(x_{1}, x_{2}, \ldots, x_{n}) \in S$ 
into two variables $(y_{0}, y_{1})$ satisfying $x_{1} = y_{0} + y_{1}$,
that is,
$U_{1}=\{ 0, 1 \}$ and $U_{i}=\{ i \}$ for $i=2,\ldots,n$
in the notation of Section~\ref{SCsplitset}.
The resulting set $T$ is given by
\begin{equation*} 
%%\label{elesplitsetdef}
 T = \{ (y_{0}, y_{1}, y_{2}, \dots , y_{n}) \in \ZZ\sp{n+1} \mid
  y_{0}+ y_{1} = x_{1}, \  y_{i} = x_{i} \ (i=2,\ldots,n), \
  (x_{1}, x_{2}, \dots , x_{n}) \in S
   \} .
\end{equation*}

To show the integral convexity of $T$, take any 
$y \in \overline{T} \subseteq \RR\sp{n+1}$.
A crucial step of the proof is to find a set of vectors $v\sp{\ell} \in T \cap N(y)$
to represent $y$ as their convex combination:
\begin{equation} \label{convcomby}
 y = \sum_{\ell} \mu_{\ell} v\sp{\ell},
\end{equation}
where $\mu_{\ell} \geq 0$ and $\sum_{\ell} \mu_{\ell} = 1$.
This means, in particular, that each 
$v\sp{\ell} \in \ZZ\sp{n+1}$
must satisfy the condition
$\lfloor y \rfloor \leq  v\sp{\ell} \leq \lceil y \rceil$,
since 
(cf., \eqref{intneighbordeffloorceil})
\[
N(y) = \{ z \in \ZZ\sp{n+1} \mid
\lfloor y_{i} \rfloor \leq  z_{i} \leq \lceil y_{i} \rceil  \ \ (i=0,1,\ldots, n) \} .
\]

We introduce notation
$\hat y= (y_{2}, \dots , y_{n}) \in \RR\sp{n-1}$.
Then  $y = (y_{0}, y_{1}, \hat y)$.
Let 
\begin{equation*}
x = (y_{0}+ y_{1}, \hat y) = (y_{0}+ y_{1}, y_{2}, \dots , y_{n}) \in \RR\sp{n}.
\end{equation*}
We have $x \in \overline{S}$.
By the integral convexity of $S$, we can represent $x$ as a convex combination
of some $u\sp{k} \in S \cap N(x)$ ($k=1,2,\ldots,m$), that is,
\begin{equation} \label{convcombx}
 x = \sum_{k=1}\sp{m} \lambda_{k} u\sp{k}
\end{equation}
with $\lambda_{k} \geq 0$ and $\sum_{k} \lambda_{k} = 1$,
where
$u\sp{k} \in S$ ($\subseteq \ZZ\sp{n}$)  and 
$\lfloor x \rfloor \leq  u\sp{k} \leq \lceil x \rceil$ for $k=1,2,\ldots,m$.
The equation \eqref{convcombx} shows
\begin{equation} \label{convcombx1}
  y_{0}+ y_{1} = \sum_{k=1}\sp{m} \lambda_{k} u\sp{k}_{1},
\qquad
  \hat y = \sum_{k=1}\sp{m} \lambda_{k} \hat u\sp{k},
\end{equation}
where 
$u\sp{k} = (u\sp{k}_{1}, \hat u\sp{k})$
with 
$u\sp{k}_{1} \in \ZZ$ and $\hat u\sp{k} \in \ZZ\sp{n-1}$.
The condition
$\lfloor x \rfloor \leq  u\sp{k} \leq \lceil x \rceil$
is equivalent to
\begin{equation} \label{boxy}
\lfloor y_{0}+ y_{1} \rfloor \leq  u\sp{k}_{1} \leq \lceil y_{0}+ y_{1} \rceil ,
\qquad
\lfloor \hat y \rfloor \leq  \hat u\sp{k} \leq \lceil \hat y \rceil .
\end{equation}
Let 
\begin{equation} \label{K0K1def}
K_{0} = \{ k \mid u\sp{k}_{1}=\lfloor y_{0} +  y_{1} \rfloor \},
\qquad 
K_{1} = \{ k \mid u\sp{k}_{1}=\lfloor y_{0} +  y_{1} \rfloor + 1 \}.
\end{equation}

Denote the fractional parts of $y_{0}$ and $y_{1}$ by
\begin{equation} \label{eta0eta1def}
 \eta_{0} = y_{0} - \lfloor y_{0} \rfloor,
\qquad
 \eta_{1} = y_{1} - \lfloor y_{1} \rfloor .
\end{equation}
We have $0 \leq \eta_{0} < 1$ and $0 \leq \eta_{1} < 1$,
from which follows
$0 \leq \eta_{0} +  \eta_{1} < 2$.
We distinguish the following cases:

Case 1:  
$0 < \eta_{0} < 1$, \ $0 < \eta_{1} < 1$, \ $\eta_{0} +  \eta_{1} < 1$ 
\quad (This is the essential case);

Case 2:  
$0 < \eta_{0} < 1$, \ $0 < \eta_{1} < 1$, \ $\eta_{0} +  \eta_{1} > 1$;

Case 3:  
$\eta_{0} = 0$ or $\eta_{1} = 0$ or $\eta_{0} +  \eta_{1} = 1$
(in addition to $0 \leq \eta_{0} < 1$ and $0 \leq \eta_{1} < 1$).

\subsubsection{Case 1:  $0 < \eta_{0} < 1$, \ $0 < \eta_{1} < 1$, \ $\eta_{0} +  \eta_{1} < 1$}
\label{SCicvsplitcase1}

In this case we have
\begin{equation} \label{case1y0y1}
\lfloor y_{0} \rfloor + 1 = \lceil  y_{0} \rceil,
\quad
\lfloor y_{1} \rfloor + 1 = \lceil  y_{1} \rceil,
\quad
\lfloor y_{0} \rfloor + \lfloor y_{1} \rfloor = \lfloor y_{0} + y_{1} \rfloor .
\end{equation}

For $k=1,2,\ldots,m$,
we define $(n+1)$-dimensional integer vectors 
$v\sp{k}$ or $\{ v\sp{k0}, v\sp{k1}  \}$ 
from the vectors 
$u\sp{k} = (u\sp{k}_{1}, \hat u\sp{k})$
in \eqref{convcombx}.
Define
\begin{align} 
 v\sp{k} &:= (\lfloor y_{0} \rfloor, \quad \ \  \lfloor y_{1} \rfloor, \quad  \  \hat u\sp{k})
\qquad \mbox{for $k \in K_{0}$},
\label{basisK0}
\\
 v\sp{k0} &:= (\lfloor y_{0} \rfloor + 1, \lfloor y_{1} \rfloor, \quad \  \hat u\sp{k})
\qquad \mbox{for $k \in K_{1}$},
\label{basisK10}
\\
 v\sp{k1} &:= (\lfloor y_{0} \rfloor , \quad \ \   \lfloor y_{1} \rfloor + 1, \hat u\sp{k})
\qquad \mbox{for $k \in K_{1}$}.
\label{basisK11}
\end{align}
We use notations
$v\sp{k} = (v\sp{k}_{0}, v\sp{k}_{1}, \hat v\sp{k})$,
$v\sp{k0} = (v\sp{k0}_{0}, v\sp{k0}_{1}, \hat v\sp{k0})$, and
$v\sp{k1} = (v\sp{k1}_{0}, v\sp{k1}_{1}, \hat v\sp{k1})$.

\medskip

Claim 1:
(i)
$v\sp{k} \in T \cap N(y)$ for $k \in K_{0}$, and
(ii)
$v\sp{k0}, v\sp{k1} \in T \cap N(y)$ for $k \in K_{1}$.
\begin{proof}[Proof of Claim 1]
(i) 
Let $k \in K_{0}$.
We have $v\sp{k} \in T$ 
since 
\[
v\sp{k}_{0} + v\sp{k}_{1} 
= \lfloor y_{0} \rfloor + \lfloor y_{1} \rfloor 
= \lfloor y_{0} + y_{1} \rfloor
= u\sp{k}_{1}.
\]
We have $v\sp{k} \in N(y)$
since
$ v\sp{k}_{i} = \lfloor y_{i} \rfloor $
for $i=0,1$
and
$\lfloor \hat y \rfloor \leq  \hat v\sp{k} = \hat u\sp{k} \leq \lceil \hat y \rceil$
by \eqref{boxy}.

(ii) 
Let $k \in K_{1}$.
We have $v\sp{k0},  v\sp{k1}  \in T$ 
since 
\[
v\sp{kj}_{0} + v\sp{kj}_{1} 
= \lfloor y_{0} \rfloor + \lfloor y_{1} \rfloor + 1 
= \lfloor y_{0} + y_{1} \rfloor + 1
= u\sp{k}_{1}
\]
for $j=0,1$.
We have $v\sp{k0} \in N(y)$
since
\[
 v\sp{k0}_{0} = \lfloor y_{0} \rfloor + 1 = \lceil  y_{0} \rceil,
\quad 
 v\sp{k0}_{1} = \lfloor y_{1} \rfloor,
\quad 
 \lfloor \hat y \rfloor \leq \hat v\sp{k0} = \hat u\sp{k} \leq \lceil \hat y \rceil
\]
by \eqref{boxy}.
Similarly, we have $v\sp{k1} \in N(y)$.
\end{proof}

We will show that we can represent $y$ as a convex combination
of the vectors in \eqref{basisK0}--\eqref{basisK11}, that is,
\begin{equation} \label{convcomby2}
 y = \sum_{k \in K_{0}} \mu_{k} v\sp{k}
    + \sum_{k \in K_{1}} ( \mu_{k0} v\sp{k0} + \mu_{k1} v\sp{k1})
\end{equation}
for some $\mu_{k}, \mu_{k0}, \mu_{k1} \geq 0$
with 
$\sum_{k \in K_{0}} \mu_{k} + \sum_{k \in K_{1}} ( \mu_{k0} + \mu_{k1} ) = 1$.
For the coefficients for $k \in K_{0}$ we take 
\begin{equation} \label{mukK0}
\mu_{k} = \lambda_{k}
\qquad
(k \in K_{0}).
\end{equation}
For the coefficients for $k \in K_{1}$ we have the following.

\medskip

Claim 2:
There exist nonnegative $\mu_{k0}$, $\mu_{k1}$ $(k \in K_{1})$
satisfying
\begin{align} 
\sum_{k \in K_{1}}  \mu_{k0} &= \eta_{0},
\label{muk0K1eta0}
\\
\sum_{k \in K_{1}}  \mu_{k1} &= \eta_{1},
\label{muk1K1eta1}
\\
 \mu_{k0} + \mu_{k1}  &= \lambda_{k}
\qquad \mbox{for each $k \in K_{1}$}.
\label{muk0muk1}
\end{align}
\begin{proof}[Proof of Claim 2]
Consider a $2 \times  |K_{1}|$ matrix (array), say, $M$
in which the first row is
$(\mu_{k0} \mid k \in K_{1})$
and 
the second row is
$(\mu_{k1} \mid k \in K_{1})$.
The conditions above say that 
the first row-sum of $M$ is equal to $\eta_{0}$,
the second row-sum is equal to $\eta_{1}$,
and the $k$-th column-sum is equal to $\lambda_{k}$.
Note that the sum
 of the row-sums
is equal to the sum 
of the column-sums,
that is,
\[
\eta_{0} + \eta_{1} =\sum_{k \in K_{1}} \lambda_{k},
\]
since
\eqref{eta0eta1def} and \eqref{case1y0y1} imply
\[
\eta_{0} + \eta_{1} 
=  y_{0} + y_{1}  - (\lfloor y_{0} \rfloor + \lfloor y_{1} \rfloor )
=  y_{0} + y_{1}  - \lfloor y_{0}  +  y_{1} \rfloor ,
\]
whereas 
\eqref{convcombx1} implies
\[
  y_{0} + y_{1}  = \lfloor y_{0}  +  y_{1} \rfloor 
  + \sum_{k \in K_{1}}  \lambda_{k} .
\]
Thus the proof of Claim 2 is reduced to showing the existence 
of a feasible (nonnegative) solution to a transportation problem.
As is well known, a feasible solution always exists and it
can be constructed by the so-called north-west corner method (or north-west rule \cite{Sch03}).
\end{proof}

The coefficients $\mu_{k}$, $\mu_{k0}, \mu_{k1}$ constructed above
have the desired properties.
Indeed, we have the following:
\begin{itemize}
\item
By \eqref{mukK0} and \eqref{muk0muk1},
they are nonnegative numbers adding up to one:
\[
\sum_{k \in K_{0}} \mu_{k} + \sum_{k \in K_{1}} ( \mu_{k0} + \mu_{k1} ) 
= \sum_{k \in K_{0}} \lambda_{k} + \sum_{k \in K_{1}} \lambda_{k} 
= 1 .
\]

\item
By \eqref{muk0K1eta0},
the first ($0$-th) component of the right-hand side of \eqref{convcomby2} is equal to
\[
\lfloor y_{0} \rfloor  + \sum_{k \in K_{1}}  \mu_{k0}
= \lfloor y_{0} \rfloor  + \eta_{0} = y_{0} .
\]

\item
By \eqref{muk1K1eta1}, 
the second component of the right-hand side of \eqref{convcomby2} is equal to
\[
\lfloor y_{1} \rfloor  + \sum_{k \in K_{1}}  \mu_{k1}
= \lfloor y_{1} \rfloor  + \eta_{1} = y_{1} .
\]

\item
By \eqref{mukK0}, \eqref{muk0muk1}, and \eqref{convcombx1},
the remaining part is equal to
\[
 \sum_{k \in K_{0}} \mu_{k} \hat v\sp{k}
    + \sum_{k \in K_{1}} ( \mu_{k0} \hat v\sp{k0} + \mu_{k1} \hat v\sp{k1})
= \sum_{k \in K_{0}} \mu_{k} \hat u\sp{k}
    + \sum_{k \in K_{1}} ( \mu_{k0} + \mu_{k1}) \hat u\sp{k}
= \sum_{k \in K_{0} \cup K_{1}} \lambda_{k} \hat u\sp{k}
= \hat y .
\]
\end{itemize}

The above argument shows the following lemma, 
which will be used in 
the proof of the splitting of integrally convex functions
in Section \ref{SCicvfnsplitprf}.

\begin{lemma} \label{LMsplitsetICprf1}
Let $y \in \overline{T}$ and 
$x = (y_{0}+ y_{1}, \hat y) = (y_{0}+ y_{1}, y_{2}, \dots , y_{n})$,
and consider an arbitrary representation 
$\displaystyle
 x = \sum_{k=1}\sp{m} \lambda_{k} u\sp{k} 
$
of $x$ as a convex combination
of $u\sp{k} \in S \cap N(x)$ $(k=1,2,\ldots,m)$,
where
$y \in \RR\sp{n+1}$ and $x, u\sp{1}, \ldots, u\sp{m} \in \RR\sp{n}$.
Assuming Case 1, the vectors 
$v\sp{k}$, $v\sp{k0}$, $v\sp{k1}$ 
defined by
\eqref{basisK0},
\eqref{basisK10},
\eqref{basisK11}
all belong to $T \cap N(y)$,
and $y$ can be represented as their convex combination as
%%\eqref{convcomby2}
\begin{equation} \label{convcomby22}
 y = \sum_{k \in K_{0}} \mu_{k} v\sp{k}
    + \sum_{k \in K_{1}} ( \mu_{k0} v\sp{k0} + \mu_{k1} v\sp{k1}) ,
\end{equation}
where
$\lambda_{k} = \mu_{k}$ for $k \in K_{0}$
and 
$\lambda_{k}  =  \mu_{k0} + \mu_{k1}$ for  $k \in K_{1}$.
%% \finbox
\end{lemma}

\subsubsection{Case 2: $0 < \eta_{0} < 1$, \ $0 < \eta_{1} < 1$, \ $\eta_{0} +  \eta_{1} > 1$}

By coordinate inversion we can reduce this case to Case 1.  Let
\[
\check S = -S,
\quad
\check T = -T,
\quad
\check y = -y,
\quad
\check x = -x.
\]
Then $\check S$ is integrally convex
and $\check T$ is an elementary splitting of $\check S$.

Denote the fractional parts of $\check y_{0}$ and $\check y_{1}$ by
\begin{equation*} 
%% \label{negeta0eta1def}
 \check \eta_{0} = \check y_{0} - \lfloor \check y_{0} \rfloor,
\qquad
 \check \eta_{1} = \check y_{1} - \lfloor \check y_{1} \rfloor .
\end{equation*}
For $i=0,1$ we have
\[
 \check \eta_{i} = -y_{i} - \lfloor -y_{i} \rfloor
   = -y_{i} + \lceil y_{i} \rceil 
= -y_{i} + ( \lfloor y_{i} \rfloor + 1)
 = 1-  \eta_{i} ,
\]
and therefore,
$0 < \check \eta_{0} < 1$, \ $0 < \check \eta_{1} < 1$, \ $\check \eta_{0} +  \check \eta_{1} < 1$.
By the argument for Case 1, we have
$\check y \in \overline{\check T \cap N(\check y)}$,
which is equivalent to
$y \in \overline{T \cap N(y)}$.

\subsubsection{Case 3: $\eta_{0} = 0$ or $\eta_{1} = 0$ or $\eta_{0} +  \eta_{1} = 1$}

In this case, $y$ lies on the boundary of the region of Case 1.
We consider a perturbation of 
$y$ in the first two components $y_{0}$ and $y_{1}$．
For an arbitrary $\varepsilon > 0$,  take 
$y\sp{\varepsilon} = ( y\sp{\varepsilon}_{0}, y\sp{\varepsilon}_{1}, y_{2},  \ldots, y_{n}) 
\in \overline{T}$
with 
$| y\sp{\varepsilon}_{i} - y_{i} | \leq \varepsilon$
($i=0,1$)
such that
$ \eta\sp{\varepsilon}_{0} = y\sp{\varepsilon}_{0} - \lfloor y\sp{\varepsilon}_{0} \rfloor$ 
and 
$ \eta\sp{\varepsilon}_{1} = y\sp{\varepsilon}_{1} - \lfloor y\sp{\varepsilon}_{1} \rfloor$
satisfy 
$0 < \eta\sp{\varepsilon}_{0} < 1$, \ $0 < \eta\sp{\varepsilon}_{1} < 1$,
and $\eta\sp{\varepsilon}_{0} +  \eta\sp{\varepsilon}_{1} < 1$.
Then we have
\[
N(y\sp{\varepsilon}) = \{ z \in \ZZ\sp{n+1} \mid
\lfloor y_{i} \rfloor \leq  z_{i} \leq \lfloor y_{i} \rfloor + 1
\ (i=0,1), \ \ 
\lfloor y_{i} \rfloor \leq  z_{i} \leq \lceil y_{i} \rceil 
\ (i=2,\ldots, n)
\},
\]
which we denote by $N(y_{*})$ 
since it does not depend on $\varepsilon$.
Note that $N(y_{*})$ is strictly larger than $N(y)$.
By the argument of Case 1,
we have $y\sp{\varepsilon} \in \overline{T \cap N(y_{*})}$.
By letting $\varepsilon \to 0$, we obtain
$y \in \overline{T \cap N(y_{*})}$
since the convex hull of $T \cap N(y_{*})$ is a closed set.
Furthermore, 
$y \in \overline{T \cap N(y_{*})}$
implies $y \in \overline{T \cap N(y)}$
in spite of the proper inclusion $N(y_{*}) \supset N(y)$.

We have completed the proof of Proposition \ref{PRsplitsetIC}.

\subsection{Proof for the splitting of integrally convex functions}
\label{SCicvfnsplitprf}

Here is a proof
of Proposition \ref{PRsplitfnIC}
concerning the splitting of an integrally convex function.
Let $g$ be an elementary splitting of an integrally convex function $f$:
\begin{equation} \label{elesplitfndef}
 g (y_{0}, y_{1}, y_{2}, \dots , y_{n})
= f(y_{0}+ y_{1}, y_{2}, \dots , y_{n})
\qquad (y \in \ZZ\sp{n+1}).
\end{equation}
The effective domain $T = \dom g$
is an elementary splitting of $S = \dom f$.

To prove the integral convexity of $g$,
it suffices,
by Theorem~\ref{THmmtt19ThA1} (if part),
to show 
that the local convex extension $\tilde{g}$ of $g$ satisfies the inequality
\begin{equation}  \label{intcnvdefg}
\tilde{g}\, \bigg(\frac{z + w}{2} \bigg) 
\leq \frac{1}{2} (g(z) + g(w))
\end{equation}
for all $z, w \in \dom g$.
Let
$\check z =  (z_{0}+ z_{1}, z_{2}, \dots , z_{n})$
and
$\check w =  (w_{0}+ w_{1}, w_{2}, \dots , w_{n})$.
By Theorem~\ref{THmmtt19ThA1} (only-if part),
the local convex extension $\tilde{f}$ of $f$ satisfies the inequality
\begin{equation*}  
%%\label{intcnvdeff}
\tilde{f}\, \bigg(\frac{\check z + \check w}{2} \bigg) 
\leq \frac{1}{2} (f(\check z) + f(\check w)) ,
\end{equation*}
%%by the integral convexity of $f$ and Theorem~\ref{THmmtt19ThA1} (only-if part),
whereas 
\[
 \frac{1}{2} (f(\check z) + f(\check w)) 
=
 \frac{1}{2} (g(z) + g(w))
\]
from \eqref{elesplitfndef}.
Therefore, the desired inequality
\eqref{intcnvdefg} 
follows from Lemma \ref{LMsplitfnICprf2} below,
where the technical result stated in Lemma \ref{LMsplitsetICprf1}
plays the crucial role in the proof.

\begin{lemma} \label{LMsplitfnICprf2}
\begin{equation}  \label{tilgtilf}
\tilde{g}\, \bigg(\frac{z + w}{2} \bigg) 
\leq
\tilde{f}\, \bigg(\frac{\check z + \check w}{2} \bigg) .
\end{equation}
\end{lemma}
\begin{proof}
Let
$y =  (z + w)/2$.
We have $y = (y_{0}, y_{1}, y_{2}, \dots , y_{n}) \in \overline{T}$.
Depending on the fractional parts 
$\eta_{0} = y_{0} - \lfloor y_{0} \rfloor$
and
$\eta_{1} = y_{1} - \lfloor y_{1} \rfloor$
of $y_{0}$ and $y_{1}$,
we have three cases as in Section \ref{SCicvsetsplitprf}.
Here we assume
Case 1 ($0 < \eta_{0} < 1$,  $0 < \eta_{1} < 1$,  $\eta_{0} +  \eta_{1} < 1$),
which is the essential case.

Let
$x = (y_{0}+ y_{1}, y_{2}, \dots , y_{n}) =(\check z + \check w)/2$.
By the definition of the local convex extension $\tilde f$,
there exist some $u\sp{k} \in S \cap N(x)$ ($k=1,2,\ldots,m$)
such that
\begin{equation*}
%% \label{convcombx2f}
 x = \sum_{k=1}\sp{m} \lambda_{k} u\sp{k} , 
\qquad
 \tilde f(x) = \sum_{k=1}\sp{m} \lambda_{k} f(u\sp{k}) ,
\end{equation*}
where $\lambda_{k} \geq 0$ and $\sum_{k} \lambda_{k} = 1$.
We now apply Lemma \ref{LMsplitsetICprf1}
in Section \ref{SCicvsplitcase1}
to obtain
\begin{equation*} 
%%\label{convcomby22again}
 y = \sum_{k \in K_{0}} \mu_{k} v\sp{k}
    + \sum_{k \in K_{1}} ( \mu_{k0} v\sp{k0} + \mu_{k1} v\sp{k1})
\end{equation*}
in  \eqref{convcomby22}.
It follows from this and the definition of the local convex extension $\tilde g$ that
\begin{equation*} 
%%\label{convcomby22g}
 \tilde g(y) \leq \sum_{k \in K_{0}} \mu_{k} g(v\sp{k})
    + \sum_{k \in K_{1}} ( \mu_{k0} g(v\sp{k0}) + \mu_{k1} g(v\sp{k1})) .
\end{equation*}
On the right-hand side we have
\begin{align*} 
& g(v\sp{k}) = f(u\sp{k})
\qquad
\mbox{for $k \in K_{0}$} ,
%%\label{mukK02gf}
\\
& g(v\sp{k0}) = g(v\sp{k1}) = f(u\sp{k}) 
\qquad \mbox{for  $k \in K_{1}$}
%%\label{muk0muk12gf}
\end{align*}
by \eqref{basisK0}, \eqref{basisK10}, and \eqref{basisK11}.
We also have
$\lambda_{k}= \mu_{k}$ $(k \in K_{0})$
and
$\lambda_{k}  =  \mu_{k0} + \mu_{k1}$ $(k \in K_{1})$.
Therefore, we have
\begin{equation*} 
%%\label{convcomby22gf}
\sum_{k \in K_{0}} \mu_{k} g(v\sp{k})
    + \sum_{k \in K_{1}} ( \mu_{k0} g(v\sp{k0}) + \mu_{k1} g(v\sp{k1}))
 = \sum_{k=1}\sp{m} \lambda_{k} f(u\sp{k}) .
\end{equation*}
From the above argument we obtain
\begin{align*} 
\tilde{g}\, \bigg(\frac{z + w}{2} \bigg) 
=
\tilde g(y)  & \leq
\sum_{k \in K_{0}} \mu_{k} g(v\sp{k})
    + \sum_{k \in K_{1}} ( \mu_{k0} g(v\sp{k0}) + \mu_{k1} g(v\sp{k1}))
\\
& = \sum_{k=1}\sp{m} \lambda_{k} f(u\sp{k}) = \tilde f(x)  
=\tilde{f}\, \bigg(\frac{\check z + \check w}{2} \bigg) ,
\end{align*}
which shows \eqref{tilgtilf}.
\end{proof}

This completes the proof of Proposition \ref{PRsplitfnIC}.

\subsection{Proof for the direct sum of multimodular sets and functions}
\label{SCmmdirsumprf}

In Section \ref{SCmmfndirsumprf}
we give a proof of Proposition \ref{PRfndirsumMult}
concerning the direct sum of multimodular functions.
Proposition \ref{PRsetdirsumMult}
for multimodular sets follows from this
as a special case for the indicator functions of sets.
In Section \ref{SCmmsetdirsumprf}
we give an alternative proof of 
Proposition \ref{PRsetdirsumMult} for multimodular sets
based on the polyhedral description of a multimodular set.

\subsubsection{Proof via discrete midpoint convexity}
\label{SCmmfndirsumprf}

We make use of
Theorem~\ref{THmmfnlnatfn}
to reduce the argument for multimodular functions to that for \Lnat-convex functions.
The direct sum operation for multimodular functions
does not correspond to the direct sum of the corresponding \Lnat-convex functions,
but to a certain new operation on variables of the \Lnat-convex functions
(cf., Lemma \ref{LMmmfn2}). 
By investigating discrete midpoint convexity
we shall show that this new operation preserves \Lnat-convexity.

First we note a simple fact about integers.

\begin{lemma} \label{LMhint}
For $a, b \in \ZZ$ we have
\begin{align} 
\left\lceil \frac{a+b}{2} \right\rceil
=  \begin{cases} 
 \left\lceil {a}/{2} \right\rceil + \left\lceil {b}/{2} \right\rceil
 & (\mbox{\rm if $a$ is even}), \\
  \left\lceil {a}/{2} \right\rceil
 + \left\lfloor {b}/{2} \right\rfloor
 & (\mbox{\rm if $a$ is odd}), \\
 \end{cases}
\label{hintup}
\\
\left\lfloor \frac{a+b}{2} \right\rfloor
=  \begin{cases} 
 \left\lfloor {a}/{2} \right\rfloor + \left\lfloor {b}/{2} \right\rfloor
 & (\mbox{\rm if $a$ is even}) ,
\\
  \left\lfloor {a}/{2} \right\rfloor
 + \left\lceil {b}/{2} \right\rceil 
 & (\mbox{\rm if $a$ is odd}). \\
 \end{cases}
\label{hintdown}
\end{align}
%% \finbox
\end{lemma}

We use variables $x \in \ZZ\sp{n_{1}}$ and $y \in \ZZ\sp{n_{2}}$ for 
multimodular functions $f_{1}$ and $f_{2}$, respectively.
To reduce the argument 
to \Lnat-convex functions,
we transform the variables 
$x$ and $y$
for multimodular functions 
to variables $p$ and $q$
for \Lnat-convex functions
through the relations $x = D_{1} p$ and $y = D_{2} q$
using matrices $D_{1}$ and $D_{2}$
of the form of \eqref{matDdef}
of sizes $n_{1}$ and $n_{2}$.
We also transform the variable 
$(x,y)$ for $f_{1} \oplus f_{2}$
to a variable $r \in \ZZ\sp{n_{1} + n_{2}}$
in a similar manner.
The following lemma reveals that $r$ is not equal to $(p,q)$,
but is equal to $(p, \ p_{*} \bm{1} + q)$,
where $p_{*}$ denotes the last component of $p$.

\begin{lemma} \label{LMmmfn2}
If $z = (x,y)$ with 
$x = D_{1} p$, \ 
$y = D_{2} q$, \  and 
$z = \tilde D r$, then
\begin{equation} \label{mmfnrpq}
r = (p, p_{*} \bm{1} + q), 
\end{equation}
where $p_{*}$ denotes the last component of $p$, i.e., $p_{*} = p_{n_{1}}$.
\end{lemma} 
\begin{proof}
The inverse of a matrix of the form \eqref{matDdef} 
is the lower triangular matrix in \eqref{matDinv},
which implies
\[
{\tilde D}\sp{-1}
= 
\left[ \begin{array}{cc}
 D_{1}\sp{-1} & O \\ {\bm 1} {\bm 1}\sp{\top}  & D_{2}\sp{-1} 
\end{array}\right] ,
\]
where ${\bm 1} {\bm 1}\sp{\top}$ is an $n_{2} \times n_{1}$ matrix. 
Since
$z = \tilde D r$,
$x = D_{1} p$, and
$y = D_{2} q$, 
we obtain
\begin{align*} 
 r 
&= {\tilde D}\sp{-1} z 
= 
\left[ \begin{array}{cc}
 D_{1}\sp{-1} & O \\ {\bm 1} {\bm 1}\sp{\top}  & D_{2}\sp{-1} 
\end{array}\right]
\left[ \begin{array}{c}
 x \\ y
\end{array}\right]
= 
\left[ \begin{array}{cc}
 D_{1}\sp{-1} & O \\ {\bm 1} {\bm 1}\sp{\top}  & D_{2}\sp{-1} 
\end{array}\right]
\left[ \begin{array}{cc}
 D_{1} & O \\ O  & D_{2} 
\end{array}\right]
\left[ \begin{array}{c}
 p \\ q
\end{array}\right]
= 
\left[ \begin{array}{cc}
 I & O \\ {\bm 1} {\bm 1}\sp{\top} D_{1}  & I 
\end{array}\right]
\left[ \begin{array}{c}
 p \\ q
\end{array}\right] .
%%%%
\end{align*}
It is easy to verify from the definition \eqref{matDdef} that 
each row of the matrix
${\bm 1} {\bm 1}\sp{\top} D_{1}$
is the $n_{1}$-dimensional unit vector
$(0,\ldots, 0, 1)$ having 1 in the last entry.
Therefore, $r = (p, p_{*} \bm{1} + q)$ as in \eqref{mmfnrpq}.
\end{proof}

Let
\[
\tilde  f = f_{1} \oplus f_{2},
\qquad
\tilde g (r) = \tilde  f (\tilde D r), 
\qquad
 g_{1} (p) = f_{1} (D_{1} p), 
\qquad
 g_{2} (q) = f_{2} (D_{2} q) .
\]
Since $f_{1}$ and $f_{2}$ are multimodular by assumption, 
$g_{1}$ and $g_{2}$ are \Lnat-convex by Theorem~\ref{THmmfnlnatfn} (only-if part).
We prove the \Lnat-convexity of  $\tilde g$
by showing its discrete midpoint convexity:
\begin{equation} \label{tilgmptconv}
 \tilde g(r) + \tilde  g(r') \geq
   \tilde g \left(\left\lceil \frac{r+r'}{2} \right\rceil\right) 
  + \tilde g \left(\left\lfloor \frac{r+r'}{2} \right\rfloor\right) 
\qquad (r, r' \in \ZZ\sp{n_{1} + n_{2}})   .
\end{equation}
The multimodularity of $\tilde f = f_{1} \oplus  f_{2}$ follows from this
by Theorem~\ref{THmmfnlnatfn} (if part).
It is noted that $\tilde g \not= g_{1} \oplus  g_{2}$ in general.

On the left-hand side of \eqref{tilgmptconv} we have 
\begin{equation} \label{tilgr}
\tilde g (r) 
= \tilde  f (\tilde D r)
= \tilde  f (x, y) 
= f_{1}(x) + f_{2}(y) = g_{1}(p) + g_{2}(q) ,
\end{equation}
where
$(x,y)\sp{\top} = \tilde D r$,
$p = D_{1}\sp{-1} x$, and 
$q = D_{2}\sp{-1} y$.
Similarly,
\begin{equation} \label{tilgrprime}
\tilde g (r') 
= \tilde  f (\tilde D r')
= \tilde  f (x', y') 
= f_{1}(x') + f_{2}(y') = g_{1}(p') + g_{2}(q') ,
\end{equation}
where
$(x',y')\sp{\top} = \tilde D r'$,
$p' = D_{1}\sp{-1} x'$, and 
$q' = D_{2}\sp{-1} y'$.

For the right-hand side of \eqref{tilgmptconv} we use
$r = (p, p_{*} \bm{1} + q)$ and
$r' = (p', p'_{*} \bm{1} + q')$
in \eqref{mmfnrpq}
to see
\begin{align} 
 \left\lceil \frac{r+r'}{2} \right\rceil
= \left( 
\left\lceil \frac{p+p'}{2} \right\rceil ,
\left\lceil \frac{p_{*}+p_{*}'}{2} \bm{1} + 
     \frac{q+q'}{2} \right\rceil 
\right),
 \label{hintup2}
\\
 \left\lfloor \frac{r+r'}{2} \right\rfloor
= \left( 
\left\lfloor \frac{p+p'}{2} \right\rfloor ,
\left\lfloor \frac{p_{*}+p_{*}'}{2} \bm{1} + 
     \frac{q+q'}{2} \right\rfloor 
\right) .
 \label{hintdown2}
\end{align}
We now apply Lemma \ref{LMhint}.

Suppose that $p_{*}+p_{*}'$ is even.  
By Lemma \ref{LMhint}
(with $a = p_{*}+p_{*}'$
and
$b = q_{i} +q'_{i}$ for $i=1,2,\ldots, n_{2}$),
 we obtain
\begin{align*} 
 \left\lceil \frac{r+r'}{2} \right\rceil
= \left( 
\left\lceil \frac{p+p'}{2} \right\rceil ,
\left\lceil \frac{p_{*}+p_{*}'}{2} \right\rceil  \bm{1} + 
   \left\lceil   \frac{q+q'}{2} \right\rceil 
\right),
%% \label{hintup2e}
\\
 \left\lfloor \frac{r+r'}{2} \right\rfloor
= \left( 
\left\lfloor \frac{p+p'}{2} \right\rfloor ,
\left\lfloor \frac{p_{*}+p_{*}'}{2} \right\rfloor  \bm{1} + 
  \left\lfloor    \frac{q+q'}{2} \right\rfloor 
\right) .
%% \label{hintdown2e}
\end{align*}
These vectors are of the form
$(\hat p, \hat p_{*} \bm{1} + \hat q)$ 
with
\[
 (\hat p, \hat q)  = 
\left (\left\lceil \frac{p+p'}{2} \right\rceil , \left\lceil \frac{q+q'}{2} \right\rceil \right) ,
\quad
\left (\left\lfloor \frac{p+p'}{2} \right\rfloor , \left\lfloor \frac{q+q'}{2} \right\rfloor \right) ,
\]
respectively.
Therefore,
\begin{align} 
 \tilde g \left( \left\lceil \frac{r+r'}{2} \right\rceil \right)
= 
g_{1} \left( \left\lceil \frac{p+p'}{2} \right\rceil \right) 
 + g_{2} \left(  \left\lceil   \frac{q+q'}{2} \right\rceil \right),
 \label{hintup4e}
\\
 \tilde g \left( \left\lfloor \frac{r+r'}{2} \right\rfloor \right)
= g_{1} \left( \left\lfloor \frac{p+p'}{2} \right\rfloor \right) 
 + g_{2} \left(  \left\lfloor   \frac{q+q'}{2} \right\rfloor \right) .
 \label{hintdown4e}
\end{align}
By 
\eqref{tilgr},
\eqref{tilgrprime},
\eqref{hintup4e},
\eqref{hintdown4e},
and the discrete midpoint convexity of 
$g_{1}$ and $g_{2}$, we obtain 
the discrete midpoint convexity of $\tilde g$ in
\eqref{tilgmptconv}.

Suppose that $p_{*}+p_{*}'$ is odd
in \eqref{hintup2} and \eqref{hintdown2}.
By Lemma \ref{LMhint}
(with $a = p_{*}+p_{*}'$
and
$b = q_{i} +q'_{i}$ for $i=1,2,\ldots, n_{2}$),
 we obtain
\begin{align*} 
 \left\lceil \frac{r+r'}{2} \right\rceil
= \left( 
\left\lceil \frac{p+p'}{2} \right\rceil ,
\left\lceil \frac{p_{*}+p_{*}'}{2} \right\rceil  \bm{1} + 
   \left\lfloor   \frac{q+q'}{2} \right\rfloor 
\right),
%% \label{hintup2o}
\\
 \left\lfloor \frac{r+r'}{2} \right\rfloor
= \left( 
\left\lfloor \frac{p+p'}{2} \right\rfloor ,
\left\lfloor \frac{p_{*}+p_{*}'}{2} \right\rfloor  \bm{1} + 
  \left\lceil    \frac{q+q'}{2} \right\rceil 
\right) .
%% \label{hintdown2o}
\end{align*}
These vectors are of the form
$(\hat p, \hat p_{*} \bm{1} + \hat q)$ 
with
\[
 (\hat p, \hat q)  = 
\left (\left\lceil \frac{p+p'}{2} \right\rceil , \left\lfloor \frac{q+q'}{2} \right\rfloor \right) ,
\quad
\left (\left\lfloor \frac{p+p'}{2} \right\rfloor , \left\lceil \frac{q+q'}{2} \right\rceil \right) ,
\]
respectively.
Therefore,
\begin{align} 
 \tilde g \left( \left\lceil \frac{r+r'}{2} \right\rceil \right)
= 
g_{1} \left( \left\lceil \frac{p+p'}{2} \right\rceil \right) 
 + g_{2} \left(  \left\lfloor   \frac{q+q'}{2} \right\rfloor \right),
 \label{hintup4o}
\\
 \tilde g \left( \left\lfloor \frac{r+r'}{2} \right\rfloor \right)
= g_{1} \left( \left\lfloor \frac{p+p'}{2} \right\rfloor \right) 
 + g_{2} \left(  \left\lceil   \frac{q+q'}{2} \right\rceil \right) .
 \label{hintdown4o}
\end{align}
By 
\eqref{tilgr},
\eqref{tilgrprime},
\eqref{hintup4o},
\eqref{hintdown4o},
and the discrete midpoint convexity of 
$g_{1}$ and $g_{2}$, we obtain 
the discrete midpoint convexity of $\tilde g$ in
\eqref{tilgmptconv}.
This completes the proof of
Proposition \ref{PRfndirsumMult}.

\subsubsection{Proof via polyhedral description}
\label{SCmmsetdirsumprf}

In this section we give an alternative proof of 
Proposition \ref{PRsetdirsumMult} for multimodular sets
based on their polyhedral descriptions.

Let
$S_{1} \subseteq \ZZ\sp{n_{1}}$
and 
$S_{2} \subseteq \ZZ\sp{n_{2}}$,
and also
$N_{1} = \{ 1,2,\ldots, n_{1} \}$
and
$N_{2} = \{ n_{1}+1, n_{1}+2,\ldots, n_{1}+n_{2} \}$.
By the polyhedral description of multimodular sets 
(cf., Theorem~\ref{THmmsetpolydes} (only-if part)), 
$S_{1}$ and $S_{2}$ can be described as
\begin{align*}
 S_{1} &= \{ x \in \ZZ\sp{n_{1}} 
 \mid  a_{I}\sp{1} \leq x(I) \leq b_{I}\sp{1}
      \ \ (\mbox{\rm $I$: consecutive interval in $N_{1}$})   \}, 
\\
 S_{2} &= \{ y \in  \ZZ\sp{n_{2}} 
 \mid  a_{J}\sp{2} \leq y(J) \leq b_{J}\sp{2}
      \ \ (\mbox{\rm $J$: consecutive interval  in $N_{2}$})   \}
\end{align*}
for some integers 
$a_{I}\sp{1}$ and $b_{I}\sp{1}$
indexed by consecutive intervals $I \subseteq N_{1}$,
and
$a_{J}\sp{2}$ and $b_{J}\sp{2}$
indexed by consecutive intervals $J \subseteq N_{2}$,
where 
$a_{I}\sp{1}, a_{J}\sp{2} \in \ZZ \cup \{ -\infty \}$ 
and $b_{I}\sp{1}, b_{J}\sp{2} \in \ZZ \cup \{ +\infty \}$.
For consecutive intervals $K \subseteq N_{1} \cup N_{2}$ define
\[
 a_{K} = 
 \begin{cases} 
 a_{K}\sp{1} & (K \subseteq N_{1}), \\
 a_{K}\sp{2} & (K \subseteq N_{2}), \\
 -\infty  & (\textrm{otherwise}),
 \end{cases}
\quad
 b_{K} = 
 \begin{cases} 
 b_{K}\sp{1} & (K \subseteq N_{1}), \\
 b_{K}\sp{2} & (K \subseteq N_{2}), \\
 +\infty & (\textrm{otherwise}).  
 \end{cases}
\]
Then we have
\begin{align*}
 S_{1} \oplus S_{2} = \{ z \in \ZZ\sp{n_{1} + n_{2}}
 \mid  a_{K} \leq z(K) \leq b_{K}
      \ \ (\mbox{\rm $K$: consecutive interval in $N_{1} \cup N_{2}$})   \}, 
\end{align*}
which shows, by Theorem~\ref{THmmsetpolydes} (if part), that 
$ S_{1} \oplus S_{2}$ is a multimodular set.

\begin{remark} \rm \label{RMmmfndirsumBySet}
The above alternative proof of Proposition \ref{PRsetdirsumMult}
for multimodular sets 
is shorter and simpler than the proof of Section \ref{SCmmfndirsumprf}
based on discrete midpoint convexity.
Furthermore, this gives an alternative proof of
Proposition \ref{PRfndirsumMult}
for multimodular functions
in the special case where $f_{1}$ are $f_{2}$ have bounded effective domains
%% Revision 2020-02-11
or they are convex-extensible.
If $\dom f_{1}$ and $\dom f_{2}$ are bounded,
then $\dom (f_{1} \oplus  f_{2}) = \dom f_{1} \oplus \dom f_{2}$ 
is also bounded, and we may use 
Theorem~\ref{THmmfnargmin} 
that characterizes a multimodular function
in terms of its minimizers.
Let $\tilde f = f_{1} \oplus f_{2}$ and
$c = (c_{1}, c_{2})$.  Then we have
\[
\argmin \tilde f[-c] = (\argmin f_{1}[-c_{1}]) \oplus (\argmin f_{2}[-c_{2}])  .
\]
Here,
$\argmin f_{1}[-c_{1}]$ and $\argmin f_{2}[-c_{2}]$
are multimodular sets
by Theorem~\ref{THmmfnargmin} (only-if part),
and their direct sum 
is also multimodular by
Proposition \ref{PRsetdirsumMult}.
Therefore, 
$\tilde f$ is a multimodular function 
by Theorem~\ref{THmmfnargmin} (if part).
%% \finbox
\end{remark}

\subsection{Proof for the splitting of multimodular sets and functions}
\label{SCmmsplitprf}

Here is a proof of 
Proposition \ref{PRsplitfnMult}
concerning the splitting of a multimodular function.
Proposition \ref{PRsplitsetMult}
for a multimodular set follows from this
as a special case for the indicator function of a set.
The proof makes use of the reduction to \Lnat-convex functions,
and it turns out that the splitting operation for a multimodular function
corresponds to introducing a dummy variable 
to an \Lnat-convex function that does not affect the function value.%
\footnote{%%%%%%%%%%%%%%%%%%%%%%
See \eqref{splitMultprf4} at the end of the proof, 
where the value of function $\hat g$ does not depend on the variable $q'_{k}$. 
} %%%% footnote %%%%%%%%%%%%%%%%% 

Let $f$ be a multimodular function and $g$ be an elementary splitting
of $f$ defined by
\[
 g(y_{1}, \dots, y_{k-1}, y'_{k}, y''_{k}, y_{k+1},  \dots , y_{n}) 
 =  f(y_{1}, \dots, y_{k-1}, y'_{k}+y''_{k}, y_{k+1},  \dots , y_{n}) .
\]
We can express this as
\[
 g(y) = f(C y) ,
\]
where
$y = (y_{1}, \dots, y_{k-1}, y'_{k}, y''_{k}, y_{k+1},  \dots , y_{n}) \in \ZZ\sp{n+1}$
and
$C = (C_{ij})$ is an $n \times (n+1)$ matrix
defined by
\begin{equation} \label{matCdef}
 C_{ij}=
 \begin{cases}
 1 & \mbox{if $1 \leq i=j \leq k$ \ or \  $k \leq i=j-1\leq n$,} \\
 0 & \mbox{otherwise}. 
 \end{cases}
\end{equation}
The correspondence of the variables is given by
\begin{equation} \label{xCy}
 x = C y ,
\end{equation}
where $x = (x_{1}, \dots, x_{k-1}, x_{k}, x_{k+1},  \dots , x_{n}) \in \ZZ\sp{n}$.

To show the multimodularity of $g$, we consider functions
$\hat f$ and $\hat g$ defined by 
\[
\hat f(p)=f(D_{n} p) ,  \qquad  \hat g(q)=g(D_{n+1} q) ,
\]
where 
$D_{n}$ is the $n \times n$ matrix of the form of \eqref{matDdef}
and
$D_{n+1}$ is the $(n+1) \times (n+1)$ matrix of the form of \eqref{matDdef}.
The correspondences of the variables are given by
\begin{equation} \label{xDpyDq}
 x = D_{n} p ,  \qquad  y=D_{n+1} q .
\end{equation}

It follows from \eqref{xCy} and \eqref{xDpyDq} that 
\[
p = D_{n}\sp{-1} x = D_{n}\sp{-1} C y =  D_{n}\sp{-1} C D_{n+1} q .
\] 
By straightforward calculation using the definitions 
\eqref{matDdef}, \eqref{matDinv}, and \eqref{matCdef}, we can obtain
that the $(i,j)$ entry of $D_{n}\sp{-1} C D_{n+1}$ is given as
\begin{equation} \label{matDnCDn1}
 (D_{n}\sp{-1} C D_{n+1})_{ij}=
 \begin{cases}
 1 & \mbox{if $1 \leq i=j \leq k-1$ \ or \  $k \leq i=j-1\leq n$,} \\
 0 & \mbox{otherwise} .
 \end{cases}
\end{equation}
Therefore, the correspondence of the variables
$p = (p_{1}, \dots, p_{k-1}, p_{k}, p_{k+1},  \dots , p_{n})$
and 
$q = (q_{1}, \dots, q_{k-1}, q'_{k}, q''_{k}, q_{k+1},  \dots , q_{n})$
is given by
\[
(p_{1}, \dots, p_{k-1}, p_{k}, p_{k+1},  \dots , p_{n})
= 
(q_{1}, \dots, q_{k-1}, q''_{k}, q_{k+1},  \dots , q_{n}).
\]
This shows that $\hat g$ does not depend on $q'_{k}$ 
and
\begin{equation}  \label{splitMultprf4}
\hat g (q_{1}, \dots, q_{k-1}, q'_{k}, q''_{k}, q_{k+1},  \dots , q_{n})
=\hat f (q_{1}, \dots, q_{k-1}, q''_{k}, q_{k+1},  \dots , q_{n}),
\end{equation}
in which $\hat f$ is \Lnat-convex.
Therefore, $\hat g$ is \Lnat-convex, which implies, 
by Theorem~\ref{THmmfnlnatfn},  that $g$ is multimodular.

%%%%%%%%%%%%%%%%%%%%%%%%%%%%%%%%%%%%

\section*{Acknowledgement}
The author thanks 
 Satoko Moriguchi,
 Akiyoshi Shioura, 
 and Akihisa Tamura for discussion and comments.
%%The author is grateful to the reviewer for constructive comments to improve the presentation. 
This work was supported by 
JSPS KAKENHI Grant Numbers 
JP26280004, JP20K11697.

%% 2019-12-03 / 2019-12-25 / 2020-02-08 / 2020-05-08

%%\newpage
%%\tableofcontents


\begin{thebibliography}{99}
\bibitem{AGH00} 
Altman, E., Gaujal, B., Hordijk, A.:
Multimodularity, convexity, and optimization properties.
Mathematics of Operations Research 
{\bf 25}, 324--347 (2000)



\bibitem{AGH03} 
Altman, E., Gaujal, B., Hordijk, A.:
Discrete-Event Control of Stochastic Networks: 
Multimodularity and Regularity.
Lecture Notes in Mathematics {\bf 1829},
Springer, Heidelberg (2003)


\bibitem{BK12} 
%% K. B{\' e}rczi and Y. Kobayashi:
B{\' e}rczi, K., Kobayashi, Y.:
An algorithm for $(n-3)$-connectivity augmentation problem:
jump system approach. 
Journal of Combinatorial Theory, 
Series B  {\bf 102}, 565--587 (2012)


\bibitem{BouC95}
%% A. Bouchet and W. H. Cunningham:
Bouchet, A., Cunningham, W.H.:
Delta-matroids, jump systems, and bisubmodular polyhedra.
SIAM Journal on Discrete Mathematics {\bf 8}, 17--32 (1995)

\bibitem{Bra10halfplane}
%% P. Br{\"a}nd{\'e}n:
Br{\"a}nd{\'e}n, P.:
Discrete concavity and the half-plane property.
SIAM Journal on Discrete Mathematics {\bf 24}, 921--933 (2010)
%%no. 3



\bibitem{Che17}
Chen, X.:
L$^{\natural}$-convexity and its applications in operations.
Frontiers of Engineering Management {\bf 4}, 283--294 (2017)






\bibitem{DW91valdel} 
%% A. W. M. Dress and W. Wenzel:
Dress, A.W.M., Wenzel, W.:
A greedy-algorithm characterization of valuated $\Delta$-matroids.
Applied Mathematics Letters {\bf 4}, 55--58 (1991)
%%% No.6







\bibitem{FT90}
%% P. Favati and F. Tardella:
Favati, P., Tardella, F.:
Convexity in nonlinear integer programming.
Ricerca Operativa {\bf 53}, 3--44 (1990)



\bibitem{Fra11book} 
%% A. Frank:
Frank, A.:
Connections in Combinatorial Optimization.
%%Oxford Lecture Series in Mathematics and Its Applications
Oxford University Press, Oxford (2011)



\bibitem{FHS17} 
%%D. Freund, S.G. Henderson,  D.B. Shmoys:
Freund, D.,   Henderson, S.G.,  Shmoys, D.B.:
Minimizing multimodular functions and
allocating capacity in bike-sharing systems.
In: Eisenbrand, F.,    Koenemann, J. (eds.)
Integer Programming and Combinatorial Optimization.
Lecture Notes in Computer Science, vol.~10328, pp.~186--198 (2017)

%%Production and Operations Management {\bf 27}, 2339--2349 (2018).
%%No. 12




\bibitem{Fuj05book}
Fujishige, S.:
Submodular Functions and Optimization,
2nd edn.
Annals of Discrete Mathematics {\bf 58},
Elsevier, Amsterdam  (2005)


\bibitem{FM00}
%% S. Fujishige and K. Murota:
Fujishige, S., Murota, K.:
Notes on L-/M-convex functions and the separation theorems.
Mathematical Programming {\bf 88}, 129--146 (2000)



\bibitem{GY94mono} 
Glasserman, P., Yao, D.D.: 
Monotone Structure in Discrete-Event Systems.
Wiley, New York (1994)



\bibitem{Haj85} 
Hajek, B.:
Extremal splittings of point processes.
Mathematics of Operations Research 
{\bf 10}, 543--556 (1985)


\bibitem{Iim10} 
%% T. Iimura:
Iimura, T.:
Discrete modeling of economic equilibrium problems.
Pacific Journal of Optimization {\bf 6}, 57--64 (2010)





\bibitem{IMT05} 
%% T. Iimura, K. Murota, and A. Tamura:
Iimura, T., Murota, K., Tamura, A.:
Discrete fixed point theorem reconsidered.
Journal of Mathematical Economics {\bf 41}, 1030--1036 (2005)



\bibitem{IW14} 
%%Takuya Iimura, Takahiro Watanabe:
Iimura, T., Watanabe,  T.:
Existence of a pure strategy equilibrium in finite symmetric games 
where payoff functions are integrally concave.
Discrete Applied Mathematics {\bf 166},  26--33 (2014)




\bibitem{KS05jump}
%% S. N. Kabadi and R. Sridhar: 
Kabadi, S.N., Sridhar, R.: 
$\Delta$-matroid and jump system.
%% {\em Journal of Applied Mathematics and Decision Sciences}, 
Journal of Applied Mathematics and Decision Sciences 
{\bf 2005}, 95--106 (2005)
%% Volume number 2005 is correct
%%doi:10.1155/JAMDS.2005.95
%%No.2


\bibitem{KM07jumplink} 
%% Y. Kobayashi and K. Murota:
Kobayashi, Y., Murota, K.:
Induction of M-convex functions by linking systems.
%% {\em Discrete Applied Mathematics}, 
Discrete Applied Mathematics {\bf 155}, 1471--1480 (2007)


\bibitem{KMT07jump} 
Kobayashi, Y., Murota, K., Tanaka, K.:
Operations on M-convex functions on jump systems.
SIAM Journal on Discrete Mathematics {\bf 21}, 107--129 (2007)



\bibitem{KST12cunconj}
%% Y. Kobayashi, J. Szab{\'o} and K. Takazawa:
Kobayashi, Y., Szab{\'o}, J., Takazawa, K.:
A proof of Cunningham's conjecture on restricted subgraphs and jump systems.
Journal of Combinatorial Theory, 
Series B {\bf 102}, 948--966 (2012)



\bibitem{KT09evenf} 
%% Y. Kobayashi and K. Takazawa:
Kobayashi, Y., Takazawa, K.:
Even factors, jump systems, and discrete convexity.
Journal of Combinatorial Theory, 
Series B {\bf 99}, 139--161 (2009)



\bibitem{KS03} 
Koole, G., van der Sluis, E.:
Optimal shift scheduling with a global service level constraint.
IIE Transactions {\bf  35}, 1049--1055 (2003)


\bibitem{LTY11nle} 
%%G. Van Der Laan, D. Talman, Z. Yang:
van der Laan, G., Talman, D., Yang, Z.:
Solving discrete systems of nonlinear equations.
European Journal of Operational Research {\bf 214}, 493--500 (2011) 




\bibitem{LY14}
Li, Q., Yu, P.:
Multimodularity and its applications in
three stochastic dynamic inventory problems.
Manufacturing \& Service Operations Management
{\bf 16}, 455--463 (2014)




\bibitem{MM19projcnvl} 
Moriguchi, S., Murota, K.:
Projection and convolution operations for integrally convex functions.
Discrete Applied Mathematics {\bf 255}, 283--298 (2019)



\bibitem{MM19multm} 
Moriguchi, S., Murota, K.:
On fundamental operations for multimodular functions.
Journal of the Operations Research Society of Japan {\bf  62}, 53--63 (2019)


\bibitem{MMTT19proxIC} 
Moriguchi, S., Murota, K.,  Tamura, A.,    Tardella, F.:
Scaling, proximity, and optimization of integrally convex functions.
Mathematical Programming {\bf 175}, 119--154 (2019)




\bibitem{MMTT20dmc}
Moriguchi, S., Murota, K.,  Tamura, A.,   Tardella, F.:
Discrete midpoint convexity.
Mathematics of Operations Research
{\bf 45}, 99--128 (2020)
%%Published Online (30 May 2019) https://doi.org/10.1287/moor.2018.0984







\bibitem{Mstein96} 
%% K. Murota:
Murota, K.:
Convexity and Steinitz's exchange property. 
Advances in Mathematics {\bf 124}, 272--311 (1996)



\bibitem{Mmax97} 
%% K. Murota:
Murota, K.:
Characterizing a valuated delta-matroid 
as a family of delta-matroids.
Journal of the Operations Research Society of Japan {\bf  40}, 565--578 (1997)




\bibitem{Mdca98} 
Murota, K.:
Discrete convex analysis. 
Mathematical Programming  {\bf 83}, 313--371 (1998)


\bibitem{Mspr2000}
Murota, K.:
Matrices and Matroids for Systems Analysis.
Springer, Berlin (2000)





\bibitem{Mdcasiam} 
Murota, K.:
Discrete Convex Analysis.
Society for Industrial and Applied Mathematics, Philadelphia (2003)


\bibitem{Mmult05} 
Murota, K.:
Note on multimodularity and L-convexity.
Mathematics of Operations Research 
{\bf 30}, 658--661 (2005)


\bibitem{Mmjump06} 
Murota, K.:
M-convex functions on jump systems: 
 A general framework for minsquare graph factor problem.
SIAM Journal on Discrete Mathematics {\bf 20}, 213--226 (2006)




\bibitem{Mbonn09} 
Murota, K.:
Recent developments in discrete convex analysis.
In: Cook, W., Lov{\'a}sz, L., Vygen, J. (eds.)
Research Trends in Combinatorial Optimization,
Chapter 11, pp.~219--260. Springer, Berlin (2009) 


\bibitem{Mdcaeco16} 
Murota, K.:
Discrete convex analysis: A tool for economics and game theory.
{Journal of Mechanism and Institution Design} 
{\bf 1}, 151--273 (2016)



\bibitem{Msurvop19} 
Murota, K.:
A survey of fundamental operations on discrete convex functions of various kinds.
Optimization Methods and Software, 
Published on-line (November 2019),

doi: 10.1080/10556788.2019.1692345


\bibitem{Mmnatjump19} 
Murota, K.:
A note on M-convex functions on jump systems.
\\
arXiv: https://arxiv.org/abs/1907.06209
(2019)



\bibitem{MS99gp} 
Murota, K., Shioura, A.:
M-convex function on generalized polymatroid.
Mathematics of Operations Research {\bf 24}, 95--105 (1999)





\bibitem{MS01rel} 
Murota, K., Shioura, A.:
Relationship of M-/L-convex functions with 
discrete convex functions by Miller and by Favati--Tardella.
Discrete Applied Mathematics {\bf 115}, 151--176 (2001)





\bibitem{MS18mnataxiom}
%% K. Murota and A. Shioura:
Murota, K., Shioura, A.:
Simpler exchange axioms for M-concave functions on generalized polymatroids.
Japan Journal of Industrial and Applied Mathematics {\bf 35}, 235--259 (2018) 



\bibitem{MTcompeq03} 
%% K. Murota and A. Tamura:
Murota, K., Tamura, A.:
Application of M-convex submodular flow problem
to mathematical economics.
Japan Journal of Industrial and Applied Mathematics
{\bf 20},  257--277 (2003)



\bibitem{MT19subgrIC} 
%% K. Murota and A. Tamura:
Murota, K., Tamura, A.:
Integrality of subgradients and biconjugates of integrally convex functions.
Optimization Letters,
https://doi.org/10.1007/s11590-019-01501-1
 (2019)


\bibitem{Oxl11}
%% J. G. Oxley:
Oxley, J.G.:
Matroid Theory. 
2nd ed., Oxford University Press, Oxford (2011)






\bibitem{Sch03}
%% A. Schrijver:
Schrijver, A.:
Combinatorial Optimization---Polyhedra and Efficiency.
Springer, Heidelberg (2003)


\bibitem{Shi96indRep}
%% A. Shioura:
Shioura, A.:
An algorithmic proof for the induction of M-convex functions through networks.
Research Reports on Mathematical and Computing Sciences.
B-317, Tokyo Institute of Technology, July 1996


\bibitem{Shi98ind}
%% A. Shioura:
Shioura, A.:
A constructive proof for the induction of M-convex functions 
through networks. 
Discrete Applied Mathematics {\bf 82}, 271--278 (1998)


\bibitem{Shi17L}
Shioura, A.:
Algorithms for L-convex function minimization:
Connection between discrete convex analysis and other research areas.
{Journal of the Operations Research Society of Japan} {\bf 60},
216--243 (2017)
%%No. 3


\bibitem{ST15jorsj} 
%% A. Shioura and A. Tamura: 
Shioura, A., Tamura, A.: 
Gross substitutes condition and discrete concavity
for multi-unit valuations: a survey.
Journal of the Operations Research Society of Japan {\bf 58}, 61--103 (2015)


\bibitem{SCB14}
Simchi-Levi, D.,    Chen, X.,   Bramel, J.:
The Logic of Logistics: Theory, Algorithms, and Applications for Logistics Management, 3rd ed.
Springer, New York (2014)
%%Springer Series in Operations Research and Financial Engineering 2014,
%%Chapter 2 (Convexity and Supermodularity), pp.~15--44


\bibitem{SW93markov}
Stidham, S., Jr., Weber, R.R.:
A survey of Markov decision models for control of networks of queues.
Queueing Systems {\bf 13},  291--314 (1993)


\bibitem{Tak14fores} 
%% K. Takazawa:
Takazawa, K.:
Optimal matching forests and valuated delta-matroids.
SIAM Journal on Discrete Mathematics {\bf 28}, 445--467 (2014)
%% No. 1



\bibitem{TT20ddmc}
Tamura, A.,    Tsurumi, K.:
Directed discrete midpoint convexity.
Japan Journal of Industrial and Applied Mathematics, 
%%{\bf **}, **--** (2020)
On-line: https://doi.org/10.1007/s13160-020-00416-0



\bibitem{dWvS00} 
de Waal, P.R., van Schuppen, J.H.:
A class of team problems with discrete action spaces: 
optimality conditions based on multimodularity.
SIAM Journal on Control and  Optimization {\bf 38}, 875--892 (2000)

\bibitem{WS87netque}
Weber, R.R., Stidham, S., Jr.:
Optimal control of service rates in networks of queues.
Advances in Applied Probability {\bf 19}, 202--218 (1987)



\bibitem{Wel76} 
%% D. J. A. Welsh: 
Welsh, D.J.A.: 
Matroid Theory.
Academic Press,  London (1976)



\bibitem{Wen93pf} 
%% W. Wenzel:
Wenzel, W.:
Pfaffian forms and  $\Delta$-matroids.
Discrete Mathematics {\bf 115}, 253--266 (1993)


\bibitem{Yan08comp}
Yang, Z.:
On the solutions of discrete nonlinear complementarity and related problems.
Mathematics of Operations Research {\bf  33}, 976--990 (2008)





\bibitem{Yan09fixpt} 
%% Z. Yang:
Yang, Z.:
Discrete fixed point analysis and its applications.
Journal of Fixed Point Theory and Applications {\bf 6}, 351--371 (2009)



\bibitem{ZL10}
Zhuang, W., Li, M.Z.F.:
A new method of proving structural properties for 
certain class of stochastic dynamic control problems.
Operations Research Letters {\bf  38}, 462--467 (2010)

\end{thebibliography}
\end{document}